\documentclass[a4paper]{amsart}

\usepackage{amssymb, dsfont, enumerate} 
\usepackage[all]{xy}
\SelectTips{cm}{}

\calclayout 

\numberwithin{equation}{section}

\theoremstyle{plain}
\newtheorem{theorem}[equation]{Theorem}
\newtheorem{proposition}[equation]{Proposition}
\newtheorem{lemma}[equation]{Lemma} 
\newtheorem{corollary}[equation]{Corollary}

\theoremstyle{definition}

\newtheorem{example}[equation]{Example}

\theoremstyle{remark} 
\newtheorem{remark}[equation]{Remark}

\newcommand{\cl}{\operatorname{cl}}

\newcommand{\colim}{\operatorname{colim}}

\renewcommand{\dim}{\operatorname{dim}}
\newcommand{\End}{\operatorname{End}}
\newcommand{\Ext}{\operatorname{Ext}}

\newcommand{\fHom}{\operatorname{\mathcal{H}\!\!\;\mathit{om}}}
\newcommand{\GrMod}{\operatorname{\mathsf{GrMod}}}

\newcommand{\hocolim}{\operatorname{hocolim}}
\newcommand{\Hom}{\operatorname{Hom}}
\newcommand{\id}{\operatorname{id}}
\newcommand{\Id}{\operatorname{Id}}
\renewcommand{\Im}{\operatorname{Im}}
\newcommand{\Inj}{\operatorname{\mathsf{Inj}}}
\newcommand{\Ker}{\operatorname{Ker}}
\newcommand{\KInj}[1]{\mathsf K(\Inj #1)}

\newcommand{\Loc}{\operatorname{Loc}}
\renewcommand{\min}{\operatorname{min}}

\newcommand{\Spec}{\operatorname{Spec}}
\newcommand{\StMod}{\operatorname{\mathsf{StMod}}}

\newcommand{\supp}{\operatorname{supp}}
\newcommand{\Supp}{\operatorname{Supp}}
\newcommand{\Thick}{\operatorname{Thick}}

\newcommand{\tot}{\operatorname{tot}}

\newcommand{\hh}[1]{H^{*}(#1)} 

\newcommand{\kos}[2]{{#1}/\!\!/{#2}}

\newcommand{\lto}{\longrightarrow}
 
\newcommand{\xra}{\xrightarrow}

\newcommand{\col}{\colon}

\newcommand{\ges}{{\scriptscriptstyle\geqslant}}

\newcommand{\lotimes}{\otimes^{\mathbf L}}

\newcommand{\bik}{Benson/Iyengar/Krause}

\def\mcU{\mathcal{U}} 
\def\mcV{\mathcal{V}}
\def\mcW{\mathcal{W}} 
 
\def\mcZ{\mathcal{Z}}

\def\sfc{\mathsf c}

\def\sfC{\mathsf C}
\def\sfD{\mathsf D} 
 
\def\sfG{\mathsf G}

\def\sfS{\mathsf S} 
\def\sfT{\mathsf T} 
\def\sfU{\mathsf U}

\def\bbZ{\mathbb Z} 

\def\one{\mathds 1}

\newcommand{\bsa}{\boldsymbol{a}}

\newcommand{\fa}{\mathfrak{a}} 

\newcommand{\fm}{\mathfrak{m}} 
\newcommand{\fp}{\mathfrak{p}}
\newcommand{\fq}{\mathfrak{q}}

\newcommand{\vf}{\varphi}
\newcommand{\gam}{\varGamma}

\def\Si{\Sigma}

\def\dd{\partial} 

\newcommand{\bloc}{{L}}

\title[Stratifying triangulated categories]{Stratifying triangulated categories}

\thanks{The research of the first and second authors was undertaken
during visits to the University of Paderborn, each supported by a
research prize from the Humboldt Foundation. The research of the
second author was partly supported by NSF grants, DMS 0602498 and DMS
0903493.}



\author{Dave Benson} 
\address{Dave Benson \\ 
Institute of Mathematics\\ 
University of Aberdeen\\ 
King's College\\ 
Aberdeen AB24 3UE\\ 
Scotland U.K.}

\author{Srikanth B. Iyengar} 
\address{Srikanth B. Iyengar\\ 
Department of Mathematics\\ 
University of Nebraska\\ 
Lincoln, NE 68588\\ 
U.S.A.}

\author{Henning Krause} 
\address{Henning Krause\\ 
Institut f\"ur Mathematik\\ 
Universit\"at Paderborn\\ 
33095 Paderborn\\ 
Germany.}
\curraddr{
Fakult\"at f\"ur Mathematik\\ 
Universit\"at Bielefeld\\ 
33501 Bielefeld\\ 
Germany}

\begin{document}

\begin{abstract}
A notion of stratification is introduced for any compactly generated
triangulated category $\mathsf T$ endowed with an action of a graded
commutative noetherian ring $R$. The utility of this notion is
demonstrated by establishing diverse consequences which follow when
$\mathsf T$ is stratified by $R$. Among them are a classification of the
localizing subcategories of $\mathsf T$ in terms of subsets of the set of
prime ideals in $R$; a classification of the thick subcategories of
the subcategory of compact objects in $\mathsf T$; and results concerning
the support of the graded $R$-module of morphisms $\mathrm{Hom}_{\mathsf T}^{*}(C,D)$
leading to analogues of the tensor product theorem for support
varieties of modular representation of groups.
 \end{abstract}

\maketitle 
\setcounter{tocdepth}{1} 
\tableofcontents

\section{Introduction}
Over the last few decades, the theory of support varieties has played an increasingly important role in various aspects of representation theory. The original context was Carlson's support varieties for modular representations of finite groups \cite{Carlson:1983a}, but the method soon spread to restricted Lie algebras
\cite{Friedlander/Parshall:1986a}, complete intersections in commutative algebra \cite{Avramov:1989a,Avramov/Buchweitz:2000a}, Hochschild cohomological support for certain finite dimensional algebras \cite{Erdmann/Holloway/Snashall/Solberg/Taillefer:2004a}, and finite group schemes \cite{Friedlander/Pevtsova:2005a,Friedlander/Pevtsova:2007a}.

One of the themes in this development has been the classification of thick or localizing subcategories of various triangulated categories of representations. This story started with Hopkins' classification \cite{Hopkins:1987a} of thick subcategories of the perfect complexes over a commutative Noetherian ring $R$, followed by Neeman's classification \cite{Neeman:1992a} of localizing subcategories of the full derived category of $R$; both involved a notion of support for complexes living in the prime ideal spectrum of $R$. Somewhat later
came the classification by Benson, Carlson and Rickard \cite{Benson/Carlson/Rickard:1997a} of the thick subcategories of the stable module category of finite dimensional representation of a finite group $G$ in terms of the spectrum of its cohomology ring.

In~\cite{\bik:2008b} we established an analogous classification theorem for the localizing subcategories of the stable module category of all representations of $G$. The strategy of proof is a series of reductions and involves a passage through various other triangulated categories admitting a tensor structure. To execute this strategy, it was important to isolate a property which would permit one to classify localizing subcategories in tensor triangulated categories, and could be tracked easily under changes of categories. This is the notion of \emph{stratification} introduced in \cite{\bik:2008b} for tensor triangulated categories, inspired by work of Hovey, Palmieri, and Strickland~\cite{Hovey/Palmieri/Strickland:1997a}. For the stable module category of $G$, this condition yields a parameterization of localizing subcategories  reminiscent of, and enhancing, Quillen stratification~\cite{Quillen:1971b} of the cohomology algebra of $G$, whence the name.

In this work we present a notion of stratification for any compactly generated triangulated category $\sfT$, and establish a number of consequences which follow when this property holds for $\sfT$. The context is that we are given an \emph{action} of a graded commutative ring $R$ on $\sfT$, namely a map from $R$ to the graded center of $\sfT$. We write $\Spec R$ for the set of homogeneous prime ideals of $R$. In \cite{\bik:2008a} we developed a theory of support for objects in $\sfT$, based on a construction of exact functors $\gam_\fp\col\sfT\to\sfT$ for each $\fp\in\Spec R$, which are analogous to local cohomology functors from commutative algebra. The support of any object $X$ of $\sfT$ is the set
\[ 
\supp_R X=\{\fp\in\Spec R\mid\gam_\fp X \ne 0\}.
\]
In this paper, we investigate in detail what is needed in order to classify localizing subcategories in this general context, in terms of the set $\Spec R$.

We separate out two essential ingredients of the process of classifying localizing subcategories. The first is the \emph{local-global principle}: it states that for each object $X$ of $\sfT$, the localizing subcategory generated by $X$ is the same as the localizing subcategory generated by the set of objects $\{\gam_\fp X\mid \fp\in \Spec R\}$. We prove that $\sfT$ has this property when, for example, the dimension of $\Spec R$ is finite.

When the local-global principle holds for $\sfT$ the problem of classifying localizing subcategories of $\sfT$ can be tackled one prime at a time. This is the content of the following result, which is part of Proposition~\ref{prop:classification}.

\begin{theorem}
\label{ithm:classification}
When the local-global principle holds for $\sfT$ there is a one-to-one correspondence between localizing subcategories of $\sfT$ and functions assigning to each $\fp\in\Spec R$ a localizing subcategory of $\gam_\fp\sfT$. The function corresponding to a localizing subcategory $\sfS$ sends $\fp$ to $\sfS\cap \gam_\fp\sfT$.
\end{theorem}

The second ingredient is that in good situations the subcategory
$\gam_{\fp}\sfT$, which consists of objects supported at $\fp$, is
either zero or contains no proper localizing subcategories. If this
property holds for each $\fp$ and the local-global principle holds,
then we say $\sfT$ is \emph{stratified} by $R$. In this case, the map
in Theorem~\ref{ithm:classification} gives a one-to-one correspondence
between localizing subcategories of $\sfT$ and subsets of
$\supp_R\sfT$, which is the set of primes $\fp$ such that
$\gam_\fp\sfT\ne 0$; see Theorem~\ref{thm:classifying}.

We draw a number of further consequences of stratification. The best
statements are available when $\sfT$, in addition to be being
stratified by $R$, is \emph{noetherian}, meaning that the $R$-module
$\End^{*}_{\sfT}(C)$ is finitely generated for each compact object $C$
in $\sfT$.

\begin{theorem}
\label{ithm:classify-thick}
If $\sfT$ is noetherian and stratified by $R$, then the map described
in Theorem \emph{\ref{ithm:classification}} gives a one-to-one
correspondence between the thick subcategories of the compact objects
in $\sfT$ and the specialization closed subsets of $\supp_R\sfT$.
\end{theorem}

This result is a rewording of Theorem~\ref{thm:classifying-thick} and
can be deduced from the classification of localizing subcategories of
$\sfT$, using an argument due to Neeman~\cite{Neeman:1992a}. We give a
different proof based on the following result, which is
Theorem~\ref{thm:intersection}.

 \begin{theorem}
\label{ithm:intersection}
If $\sfT$ is noetherian and stratified by $R$, then for each pair of
compact objects $C,D$ in $\sfT$ there is an equality
\[
\supp_{R}\Hom^{*}_{\sfT}(C,D) = \supp_{R}C\cap \supp_{R}D\,.
\]
When in addition $R^{i}=0$ holds for $i<0$, one has
$\Hom^{n}_{\sfT}(C,D)=0$ for $n\gg 0$ if and only if
$\Hom^{n}_{\sfT}(D,C)=0$ for $n\gg 0$.
\end{theorem}

The statement of this theorem is inspired by an analogous statement for modules over complete intersection local rings, due to Avramov and Buchweitz~\cite{Avramov/Buchweitz:2000a}. A stratification theorem is
not yet available in this context; see however \cite{Iyengar:2009a}. 

The stratification condition also implies that Ravenel's `telescope conjecture'~\cite{Ravenel:1987a}, sometimes called the `smashing conjecture', holds for $\sfT$.

\begin{theorem}
\label{ithm:smash}
If $\sfT$ is noetherian and stratified by $R$ and $L\col\sfT\to\sfT$
is a localization functor that preserves arbitrary coproducts, then
the localizing subcategory $\Ker L$ is generated by objects that are
compact in $\sfT$.
\end{theorem}

This result is contained in Theorem~\ref{thm:smash}, which establishes
also a classification of localizing subcategories of $\sfT$ that are
also closed under products. Another application,
Corollary~\ref{cor:rickard}, addresses a question of Rickard. If
$\sfS$ is a localizing subcategory of $\sfT$, write ${^\perp}\sfS$ for
the full subcategory of objects $X$ such that there are no nonzero
morphisms from $X$ to any object in $\sfS$.

\begin{theorem}
\label{ithm:jeremy}
Suppose that $\sfT$ is noetherian and stratified by $R$, and that
$\sfS$ is a localizing subcategory of $\sfT$.  Then ${^\perp}\sfS$ is
the localizing subcategory corresponding to the set of primes
$\{\fp\in\Spec R\mid \mcV(\fp)\cap \supp_{R}\sfS=\emptyset\}$.
\end{theorem}

In Section~\ref{sec:Tensor triangulated categories} we consider the case when $\sfT$ has a structure of a tensor triangulated category compatible with the $R$-action, and discuss a notion of stratification suitable for this context. A noteworthy feature is that the analogue of the local-global principle always holds, so stratification concerns only whether each $\gam_{\fp}\sfT$ is minimal as tensor ideal localizing subcategory. When this property holds one has the following analogue of the tensor product theorem of modular representation theory as described in \cite[Theorem 10.8]{Benson/Carlson/Rickard:1996a}; cf. also Theorem~\ref{ithm:intersection}.

\begin{theorem}
Let $\sfT$ be a tensor triangulated category with a canonical $R$-action. If $R$ stratifies $\sfT$, then for any objects $X,Y$ in $\sfT$ there is an equality
\[
\supp_{R}(X\otimes Y) = \supp_{R} X \cap \supp_{R}Y\,.
\] 
\end{theorem}

This result reappears as Theorem~\ref{thm:ttintersection}. One can establish analogues of other results discussed above for tensor triangulated categories, but we do not do so; the arguments required are the same, and in any case, many of these results appear already in \cite{\bik:2008b}, at least for triangulated categories associated to modular representations of finite groups.

Most examples of stratified triangulated categories that appear in this work are imported from elsewhere in the literature. The one exception is the derived category of differential graded modules over any graded-commutative noetherian ring $A$. In Section~\ref{sec:Formal differential graded algebras} we verify that this triangulated category is stratified by the canonical $A$-action, building on arguments from \cite[\S5]{\bik:2008b} which dealt with the case $A$ is a graded polynomial algebra over a field.

There are interesting classes of triangulated categories which cannot be stratified via a ring action, in the sense explained above; see  Example~\ref{ex:quiver}.  On the other hand, there are important contexts where it is reasonable to expect stratification, notably, modules over cocommutative Hopf algebras and modules over the Steenrod algebra, where analogues of Quillen stratification have been proved by Friedlander and Pevtsova~\cite{Friedlander/Pevtsova:2005a} and Palmieri \cite{Palmieri:1999} respectively. One goal of \cite{\bik:2008a} and the present work is to pave the way to such results.

\subsection*{Acknowledgments}
It is our pleasure to thank Zhi-Wei Li for a critical reading of an earlier version of this manuscript.

\section{Local cohomology and support}
\label{sec:Local cohomology and support}

The foundation for this article is the work in \cite{\bik:2008a} where
we constructed analogues of local cohomology functors and support from
commutative algebra for triangulated categories.  In this section we
further develop these ideas, as required, and along the way recall
basic notions and constructions from \emph{op.~cit.}

\medskip

\emph{Henceforth $R$ denotes a graded-commutative noetherian ring and $\sfT$ a compactly generated $R$-linear triangulated category with arbitrary coproducts.}

\medskip

We begin by explaining what this means.

\subsection*{Compact generation}
Let $\sfT$ be a triangulated category admitting arbitrary coproducts.
A \emph{localizing subcategory} of $\sfT$ is a full triangulated
subcategory that is closed under taking coproducts.  We write
$\Loc_\sfT(\sfC)$ for the smallest localizing subcategory containing a
given class of objects $\sfC$ in $\sfT$, and call it the localizing
subcategory \emph{generated} by $\sfC$.

An object $C$ in $\sfT$ is \emph{compact} if the functor
$\Hom_{\sfT}(C,-)$ commutes with all coproducts; we write
$\sfT^{\sfc}$ for the full subcategory of compact objects in
$\sfT$. The category $\sfT$ is \emph{compactly generated} if it is
generated by a set of compact objects.

We recall some facts concerning localization functors; see, for
example, \cite[\S3]{Benson/Iyengar/Krause:2008a}.

\subsection*{Localization}
A \emph{localization functor} $L\col\sfT\to\sfT$ is an exact functor that admits a natural transformation $\eta\col\Id_\sfT\to L$, called \emph{adjunction}, such that $L(\eta X)$ is an isomorphism and $L(\eta X)=\eta (LX)$ for all objects $X\in\sfT$.  A localization functor $L\col\sfT\to\sfT$ is essentially uniquely determined by the
corresponding full subcategory
\[
\Ker L=\{X\in\sfT\mid LX=0\}\,.
\]
This means that if $L'$ is a localization functor with $\Ker L \subseteq \Ker L'$ and $\eta'$ is its adjunction, then there is a unique morphism $\iota\col L\to L'$ such that $\iota\eta=\eta'$. Given such a localization functor $L$, the natural transformation $\Id_\sfT\to L$ induces for each object $X$ in $\sfT$ a natural exact
\emph{localization triangle}
\begin{equation*}
\gam X\lto X\lto \bloc X\lto
\end{equation*}
This exact triangle gives rise to an exact functor
$\gam\col\sfT\to\sfT$ with
\[
\Ker L = \Im\gam \quad\text{and}\quad \Ker\gam =\Im L\,.
\]
Here $\Im F$, for any functor $F\col\sfT\to\sfT$, denotes the
\emph{essential image}: the full subcategory of $\sfT$ formed by
objects $\{X\in\sfT\mid X\cong FY\text{ for some $Y$ in $\sfT$}\}$.

The next lemma provides the existence of localization functors with
respect to a fixed localizing subcategory; see
\cite[Theorem~2.1]{Neeman:1992b} for the special case that the
localizing subcategory is generated by compact objects.

\begin{lemma}
\label{lem:loc-set}
Let $\sfT$ be a compactly generated triangulated category. If a
localizing subcategory $\sfS$ of $\sfT$ is generated by a set of
objects, then there exists a localization functor $L\col\sfT\to\sfT$
with $\Ker L=\sfS$.
\end{lemma}

\begin{proof}
In \cite[Corollary~4.4.3]{Neeman:2001a} it is shown that the
collection of morphisms between each pair of objects in the Verdier
quotient $\sfT/\sfS$ form a set.  The quotient functor $Q\col
\sfT\to\sfT/\sfS$ preserves coproducts, and a standard argument based
on Brown's representability theorem \cite{Keller:1994a,Neeman:1996}
yields an exact right adjoint $Q_\rho$. Note that $Q_\rho$ is fully
faithful; see \cite[Proposition~I.1.3]{Gabriel/Zisman:1967a}. It
follows that the composite $L=Q_\rho Q$ is a localization functor
satisfying $\Ker L=\sfS$; see
\cite[Lemma~3.1]{Benson/Iyengar/Krause:2008a}.
\end{proof}

\subsection*{Central ring actions}
Let $R$ be a graded-commutative ring; thus $R$ is $\bbZ$-graded and
satisfies $r\cdot s=(-1)^{|r||s|}s\cdot r$ for each pair of
homogeneous elements $r,s$ in $R$.  We say that the triangulated
category $\sfT$ is \emph{$R$-linear}, or that $R$ acts on $\sfT$, if
there is a homomorphism $R\to Z^*(\sfT)$ of graded rings, where $
Z^*(\sfT)$ is the graded center of $\sfT$. In this case, for all
objects $X,Y\in\sfT$ the graded abelian group
\[ 
\Hom^*_\sfT(X,Y)=\bigoplus_{i\in\bbZ}\Hom_\sfT(X,\Si^i Y)
\]
carries the structure of a graded $R$-module.

\subsection*{Support} 
From now on, $R$ denotes a graded-commutative noetherian ring and $\sfT$ a compactly generated $R$-linear triangulated category with arbitrary coproducts.

We write $\Spec R$ for the set of homogeneous prime ideals of $R$.  Given a homogeneous ideal $\fa$ in $R$, we set
\[
\mcV(\fa) = \{\fp\in\Spec R\mid \fp\supseteq \fa\}\,.
\] 

Let $\fp$ be a point in $\Spec R$ and $M$ a graded $R$-module. We
write $M_{\fp}$ for the homogeneous localization of $M$ at $\fp$. When
the natural map of $R$-modules $M\to M_{\fp}$ is bijective $M$ is said
to be $\fp$-local. This condition is equivalent to:
$\supp_{R}M\subseteq \{\fq\in\Spec R\mid \fq\subseteq\fp\}$, where
$\supp_{R}M$ is the support of $M$. The module $M$ is $\fp$-torsion if
each element of $M$ is annihilated by a power of $\fp$; equivalently,
if $\supp_{R}M\subseteq \mcV(\fp)$; see \cite[\S2]{\bik:2008a} for
proofs of these assertions.

The \emph{specialization closure} of a subset $\mcU$ of $\Spec R$ is
the set
\[
\cl\mcU=\{\fp\in\Spec R\mid\text{there exists $\fq\in\mcU$ with
$\fq\subseteq \fp$}\}\,.
\] 
The subset $\mcU$ is \emph{specialization closed} if $\cl\mcU=\mcU$;
equivalently, if $\mcU$ is a union of Zariski closed subsets of $\Spec
R$. For each specialization closed subset $\mcV$ of $\Spec R$, we
define the full subcategory of $\sfT$ of $\mcV$-torsion objects as
follows:
\[
\sfT_\mcV= \{X\in\sfT\mid\Hom^*_\sfT(C,X)_\fp= 0\text{ for all }
C\in\sfT^c,\, \fp\in \Spec R\setminus \mcV\}\,.
\]
This is a localizing subcategory and there exists a localization
functor $L_\mcV\col\sfT\to\sfT$ such that $\Ker L_\mcV=\sfT_\mcV$; see
\cite[Lemma 4.3, Proposition 4.5]{\bik:2008a}. For each object $X$ in
$\sfT$ the adjunction morphism $X\to L_\mcV X$ induces the exact
localization triangle
\begin{equation}
\label{eq:locseq}
\gam_{\mcV}X\lto X\lto \bloc_{\mcV}X\lto\,.
\end{equation}
This exact triangle gives rise to an exact \emph{local cohomology
functor} $\gam_\mcV\col\sfT\to\sfT$.  In \cite{\bik:2008a} we
established a number of properties of these functors, to facilitate
working with them. We single out one that is used frequently in this
work: They commute with all coproducts in $\sfT$; see \cite[Corollary
6.5]{\bik:2008a}.

For each $\fp$ in $\Spec R$ set $\mcZ(\fp)=\{\fq\in\Spec R\mid
\fq\not\subseteq \fp\}$, so $\mcV(\fp)\setminus\mcZ(\fp)=\{\fp\}$, and
\[
X_\fp=L_{\mcZ(\fp)}X\quad\text{for each $X\in\sfT$.}
\]
The notation is justified by the next result which enhances
\cite[Theorem 4.7]{\bik:2008a}.

\begin{proposition}
\label{prop:cohlocal}
Let $\fp$ be a point in $\Spec R$ and $X,Y$ objects in $\sfT$. The
$R$-modules $\Hom_\sfT^*(X,Y_{\fp})$ and $\Hom_\sfT^*(X_{\fp},Y)$ are
$\fp$-local, so the adjunction morphism $Y\to Y_{\fp}$ induces a
unique homomorphism of $R$-modules
\[
\Hom_\sfT^*(X,Y)_{\fp}\lto\Hom_\sfT^*(X,Y_{\fp})\,.
\]
This map is an isomorphism if $X$ is compact.
\end{proposition}

\begin{proof}
The last assertion in the statement is \cite[Theorem
  4.7]{\bik:2008a}. It implies that the $R$-module
$\Hom_{\sfT}^{*}(C,Y_{\fp})$ is $\fp$-local for each compact object
$C$ in $\sfT$. It then follows that $\Hom_\sfT^*(X,Y_{\fp})$ is
$\fp$-local for each object $X$, since $X$ is in the localizing
subcategory generated by the compact objects, and the subcategory of
$\fp$-local modules is closed under taking products, kernels,
cokernels and extensions; see \cite[Lemma 2.5]{\bik:2008a}.

At this point we know that $\End_{\sfT}^{*}(X_{\fp})$ is $\fp$-local,
and hence so is $\Hom^{*}_{\sfT}(X_{\fp},Y)$, since the $R$-action on
it factors through the homomorphism $R\to \End^{*}_{\sfT}(X_{\fp})$.
\end{proof}

Consider the exact functor $\gam_{\fp}\col\sfT\to\sfT$ obtained by
setting
\[
\gam_{\fp}X= \gam_{\mcV(\fp)}(X_{\fp})\,.
\]
for each object $X$ in $\sfT$.  The essential image of the functor
$\gam_\fp$ is denoted by $\gam_\fp\sfT$, and an object $X$ in $\sfT$
belongs to $\gam_\fp\sfT$ if and only if $\Hom^{*}_{\sfT}(C,X)$ is
$\fp$-local and $\fp$-torsion for every compact object $C$; see
\cite[Corollary~4.10]{\bik:2008a}. From this it follows that
$\gam_\fp\sfT$ is a localizing subcategory.

The \emph{support} of an object $X$ in $\sfT$ is a subset of $\Spec R$
defined as follows:
\[
\supp_{R} X=\{\fp\in\Spec R\mid\gam_{\fp}X\ne 0\}\,.
\] 

In addition to properties of the functors $\gam_{\mcV}$ and
$L_{\mcV}$, and support, given in \cite{\bik:2008a}, we require also
the following ones.

\begin{lemma}
\label{lem:gamp}
Let $\mcV\subseteq\Spec R$ be a specialization closed subset and $\fp\in\Spec R$.
Then for each object $X$ in $\sfT$ one has
\[
\gam_{\fp}(\gam_{\mcV}X) \cong
\begin{cases}
\gam_{\fp}X & \text{when $\fp\in\mcV$,} \\ 0 & \text{otherwise,}
\end{cases}
\qquad\text{and}\qquad \gam_{\fp}(L_{\mcV}X) \cong
\begin{cases}
\gam_{\fp}X & \text{when $\fp\not\in\mcV$,} \\ 0 & \text{otherwise.}
\end{cases}
\]
\end{lemma}

\begin{proof}
Apply the exact functor $\gam_\fp$ to the exact triangle $\gam_\mcV X\to X \to L_\mcV X\to$. The assertion then follows from the fact that either  $\gam_{\fp}(L_{\mcV}X)=0$ (and this happens precisely when $\fp\in\mcV$) or $\gam_{\fp}(\gam_{\mcV}X)=0$; see \cite[Theorem~5.6]{\bik:2008a}.
\end{proof}

Further results involve a useful construction from \cite[5.10]{\bik:2008a}.

\subsection*{Koszul objects}
Let $r\in R$ be a homogeneous element of degree $d$ and $X$ an object in $\sfT$. We denote $\kos Xr$ any object that appears in an exact triangle
\begin{equation}
\label{eq:kos}
X\stackrel{r}\lto \Si^{d}X\lto \kos X r \lto 
\end{equation}
and call it a \emph{Koszul object of $r$ on $X$}; it is well defined up to (nonunique) isomorphism. Given a homogeneous ideal $\fa$ in $R$ we write $\kos X{\fa}$ for any Koszul object obtained by iterating the construction above with respect to some finite sequence of generators for $\fa$. This object may depend on the choice of the generating sequence for $\fa$, but one has the following uniqueness statement; see also Proposition~\ref{prop:kosproperties}(2).

\begin{lemma}
\label{lem:kos-props}
Let $\fa$ be a homogenous ideal in $R$. Each object $X$ in $\sfT$ satisfies
\[
\supp_{R}(\kos X{\fa})=\mcV(\fa)\cap \supp_{R}X\,.
\]
\end{lemma}

\begin{proof} We verify the claim for $\fa=(r)$; an obvious iteration gives the general result.

Fix a point $\fp$ in $\Spec R$ and a compact object $C$ in $\sfT$. Applying the exact functor $\gam_{\fp}$ to the exact triangle \eqref{eq:kos}, and then the functor $\Hom^*_{\sfT}(C,-)$ results in an exact sequence of $R$-modules
\begin{multline*}
\Hom^*_{\sfT}(C,\gam_{\fp}X)\stackrel{\mp
r}\lto\Hom^*_{\sfT}(C,\gam_{\fp}X)[d]\lto\\ \lto
\Hom^*_{\sfT}(C,\gam_{\fp}(\kos Xr))\lto
\Hom^*_{\sfT}(C,\gam_{\fp}X)[1]\stackrel{\pm r}
\lto\Hom^*_{\sfT}(C,\gam_{\fp}X)[d+1]\, .
\end{multline*}
Set $H=\Hom^{*}_{\sfT}(C,\gam_{\fp}X)$. The $R$-module $H$ is $\fp$-local and $\fp$-torsion, see \cite[Corollary 4.10]{\bik:2008a}, and this is used as follows. If $\Hom^{*}_{\sfT}(C,\gam_{\fp}(\kos Xr))\ne 0$ holds, then $H\ne 0$ and $r\in \fp$ since $H$ is $\fp$-local. On the other hand, $H\ne 0$ and $r\in \fp$ implies that $\Hom^{*}_{\sfT}(C,\gam_{\fp}(\kos Xr))\ne 0$ since $H$ is $\fp$-torsion. This implies the desired equality.
\end{proof}

The result below is \cite[Proposition 3.5]{\bik:2008b}, except that
there $\sfG$ is assumed to consist of a single object. The argument is
however the same, so we omit the proof.

\begin{proposition}
\label{prop:comp}
\pushQED{\qed} 
Let $\sfG$ be a set of compact objects which generate
$\sfT$, and let $\mcV$ be a specialization closed subset of $\Spec
R$. For any decomposition $\mcV=\bigcup_{i\in I}\mcV(\fa_{i})$ where
each $\fa_{i}$ is an ideal in $R$, there are equalities
\[
\sfT_\mcV =\Loc_\sfT (\{\kos C{\fa_{i}} \mid C\in\sfG,\,i\in I\})
=\Loc_\sfT (\{\gam_{\mcV(\fa_{i})}C \mid C\in\sfG,\,i\in
I\})\,. \qedhere
\]
\end{proposition}

An element $r\in R^{d}$ is invertible on an $R$-module $M$ if the map
$M\xra{r}M[d]$ is an isomorphism. In the same vein, we say $r$ is
\emph{invertible} on an object $X$ in $\sfT$ if the natural morphism
$X\xra{r}\Si^{d}X$ is an isomorphism; equivalently, if $\kos Xr$ is
zero.

\begin{lemma}
\label{lem:invertible}
Let $X$ be an object in $\sfT$ and $\mcV\subseteq\Spec R$ a
specialization closed subset. Each element $r\in R$ with
$\mcV(r)\subseteq \mcV$ is invertible on $\bloc_{\mcV}X$, and hence on
the $R$-modules $\Hom^{*}_{\sfT}(\bloc_{\mcV}X,Y)$ and
$\Hom^{*}_{\sfT}(Y,\bloc_{\mcV}X)$, for any object $Y$ in $\sfT$.
\end{lemma}

\begin{proof}
From \cite[Theorem 5.6]{\bik:2008a} and Lemma~\ref{lem:kos-props} one
gets equalities
\[
\supp_{R}\bloc_{\mcV}(\kos Xr)= \mcV(r)\cap \supp_{R}X \cap (\Spec R\setminus
\mcV(r)) = \emptyset\,.
\]
Therefore $\bloc_{\mcV}(\kos Xr)=0$, by \cite[Theorem
5.2]{\bik:2008a}. Applying $\bloc_{\mcV}$ to the exact triangle
\eqref{eq:kos} yields an isomorphism $\bloc_{\mcV}
X\xra{r}\Si^{|r|}\bloc_{\mcV}X$, which is the first part of the
statement. Applying $\Hom^{*}_{\sfT}(-,Y)$ and $\Hom^{*}_{\sfT}(Y,-)$
to it gives the second part.
\end{proof}

\subsection*{Homotopy colimits}
Let $X_{1}\xra{f_{1}} X_2 \xra{f_{2}} X_3 \xra{f_{3}}\cdots$ be a sequence of morphisms in $\sfT$.  Its \emph{homotopy colimit}, denoted $\hocolim X_n$, is defined by an exact triangle
\[
\bigoplus_{n\ges1} X_{n}\stackrel{\theta}\lto \bigoplus_{n\ges1}
X_{n}\lto \hocolim X_{n}\lto
\]
where $\theta$ is the map $(\id-f_{n})$; see \cite{Bokstedt/Neeman:1993a}. 

Now fix a homogeneous element $r\in R$ of degree $d$. For each $X$ in $\sfT$ and each integer $n$ set $X_n=\Si^{nd}X$ and consider the commuting diagram
\[
\xymatrix{
X\ar@{=}[r]\ar[d]^{r^{}} &X\ar@{=}[r]\ar[d]^{r^{2}} &X\ar@{=}[r]\ar[d]^{r^{3}} &\cdots\\
X_1\ar[d]\ar[r]^r        &X_2\ar[d]\ar[r]^r         &X_3\ar[d]\ar[r]^r         &\cdots\\ \kos
Xr\ar[r]                 &\kos Xr^{2}\ar[r]         &\kos Xr^{3}\ar[r]         &\cdots }
\] 
where each vertical sequence is given by the exact triangle defining $\kos Xr^{n}$, and the morphisms in the last row are the (non-canonical) ones induced by the commutativity of the upper squares. The gist of the next result is that the homotopy colimits of the horizontal sequences in the diagram compute $\bloc_{\mcV(r)}X$ and $\gam_{\mcV(r)}X$.

\begin{proposition}
\label{prop:hocolim-local}
Let $r\in R$ be a homogeneous element of degree $d$. For each $X$ in $\sfT$ the adjunction morphisms $X\to L_{\mcV(r)} X$ and $\gam_{\mcV(r)}X\to X$ induce isomorphisms
\[
\hocolim X_n\stackrel{\sim}\lto L_{\mcV(r)}X\quad\text{and}\quad
\hocolim \Si^{-1}(\kos Xr^{n})\stackrel{\sim}\lto \gam_{\mcV(r)}X\, .
\]
\end{proposition}
\begin{proof}
Applying the functor $\gam_{\mcV(r)}$ to the middle row of the diagram above yields a sequence of morphisms
$\gam_{\mcV(r)}X_{1} \to \gam_{\mcV(r)}X_{2} \to \cdots$. For each compact object $C$ in $\sfT$, this  induces a sequence of morphisms of $R$-modules
\[
\Hom^{*}_{\sfT}(C,\gam_{\mcV(r)}X_{1})\stackrel{g_{1}}\lto 
\Hom^{*}_{\sfT}(C,\gam_{\mcV(r)}X_{2})\stackrel{g_{2}}\lto \cdots
\]
Each $R$-module $\Hom_{\sfT}^{*}(C,\gam_{\mcV(r)}X_n)$ is $(r)$-torsion and, identifying this module with $\Hom_{\sfT}^{*}(C,\gam_{\mcV(r)}X)[nd]$, the map $g_{n}$ is given by multiplication with $r$. Thus the colimit of the sequence above, in the category of $R$-modules, satisfies:
\begin{equation}
\label{eq:colim}
\colim \Hom_\sfT^*(C,\gam_{\mcV(r)} X_n) =0
\end{equation}

Applying the functor $\bloc_{\mcV(r)}$ to the canonical morphism $\phi\col X\to\hocolim X_n$ yields the following commutative square.
\[
\xymatrix{ 
X\ar[rr]^-\phi\ar[d]_{\eta X} &&\hocolim X_n\ar[d]^{\eta\hocolim X_n}\\ 
\bloc_{\mcV(r)} X\ar[rr]^-{\bloc_{\mcV(r)}\phi}&&\bloc_{\mcV(r)}\hocolim X_n}
\]
The morphism $\eta \hocolim X_n$ is an isomorphism since $\gam_{\mcV(r)}\hocolim X_n=0$. The equality holds because, for each compact object $C$, there is a chain of isomorphisms
\begin{align*}
\Hom_\sfT^*(C,\gam_{\mcV(r)}\hocolim X_n) 
&\cong \Hom_\sfT^*(C,\hocolim\gam_{\mcV(r)} X_n)\\
&\cong \colim\Hom_\sfT^*(C,\gam_{\mcV(r)} X_n)\\
&\cong 0
\end{align*}
where the second one holds because $C$ is compact and the last one is by \eqref{eq:colim}.

On the other hand, $\bloc_{\mcV(r)}\phi$ is an isomorphism, since
\[
\bloc_{\mcV(r)}\hocolim X_n\cong \hocolim \bloc_{\mcV(r)}X_n
\]
and $r$ is invertible on $\bloc_{\mcV(r)}X$, by Lemma~\ref{lem:invertible}.  Thus $\hocolim X_n\cong
\bloc_{\mcV(r)}X$.

Now consider the canonical morphism $\psi\colon \hocolim \Si^{-1}(\kos Xr^{n})\to X$. Applying the functor $\gam_{\mcV(r)}$ to it yields a commutative square:
\[
\xymatrix{ 
\gam_{\mcV(r)}\hocolim \Si^{-1}(\kos Xr^{n}) \ar[d]_-{\theta \hocolim \Si^{-1}(\kos Xr^{n})}
     \ar[rr]^-{\gam_{\mcV(r)}\psi}&&\gam_{\mcV(r)}X\ar[d]^-{\theta X}\\ 
 \hocolim \Si^{-1}(\kos Xr^{n})\ar[rr]^-{\psi}&&X}
\]
By \cite[Lemma~5.11]{\bik:2008a}, each $\kos Xr^n$ is in $\sfT_{\mcV(r)}$ and hence so is $\hocolim \Si^{-1}(\kos Xr^{n})$. Thus the morphism $\theta \hocolim \Si^{-1}(\kos Xr^{n})$ is an isomorphism. It remains to
show that $\gam_{\mcV(r)}\psi$ is an isomorphism; equivalently, that the map $\Hom_\sfT^*(C,\gam_{\mcV(r)}\psi)$ is an isomorphism for each compact object $C$.  

The exact triangle $X\to X_n\to\kos Xr^n\to$ induces an exact sequence of $R$-modules:
\begin{multline*}
\Hom_\sfT^*(C,\Si^{-1}\gam_{\mcV(r)}X_n)\lto\Hom_\sfT^*(C,\Si^{-1}(\gam_{\mcV(r)} \kos Xr^{n}))\lto\\ 
\lto \Hom_\sfT^*(C,\gam_{\mcV(r)}X)\lto\Hom_\sfT^*(C,\gam_{\mcV(r)}X_n)
\end{multline*}
In view of \eqref{eq:colim}, passing to their colimits yields that $\Hom_\sfT^*(C,\gam_{\mcV(r)}\psi)$ is an isomorphism, as desired.
\end{proof}

\begin{proposition}
\label{prop:kosproperties}
Let $\fa$ be an ideal in $R$. For each object $X$ in $\sfT$ the
following statements hold:
\begin{enumerate}[\quad\rm(1)]
\item $\kos X{\fa}$ is in $\Thick_{\sfT}(\gam_{\mcV({\fa})}X)$ and
  $\gam_{\mcV({\fa})}X$ is in $\Loc_{\sfT}(\kos X{\fa})$;
\item $\Loc_{\sfT}(\kos X{\fa})= \Loc_{\sfT}(\gam_{\mcV({\fa})}X)$;
\item $\gam_{\mcV(\fa)}X$ and $L_{\mcV(\fa)}X$ are in
  $\Loc_{\sfT}(X)$.
\end{enumerate}
\end{proposition}

\begin{proof}
(1) By construction $\kos X{\fa}$ is in $\Thick_{\sfT}(X)$. As $\gam_{\mcV(\fa)}$ is an exact functor, one obtains that $\gam_{\mcV(\fa)}(\kos X{\fa})$ is in $\Thick_{\sfT}(\gam_{\mcV(\fa)}X)$. This justifies the first claim in (1), since $\kos X{\fa}$ is in $\sfT_{\mcV(\fa)}$ by \cite[Lemma~5.11]{\bik:2008a}.

Now we verify that $\gam_{\mcV(\fa)}X$ is in the localizing subcategory generated by $\kos X{\fa}$.

Consider the case where $\fa$ is generated by a single element, say $a$. 

\smallskip
\textit{Claim}: $\kos X{a^{n}}$ is in $\Thick_{\sfT}(\kos Xa)$, for each $n\geq 1$.
\smallskip

Indeed, this is clear for $n=1$. For any $n\geq 1$, the composition of maps
\[
X\stackrel{a^{n}}\lto \Si^{n|a|}X \stackrel{a}\lto \Si^{(n+1)|a|}X
\]
yields, by the octahedral axiom, an exact triangle
\[
\kos{X}{a^{n}}\lto \kos{X}{a^{n+1}}\lto \Si^{n|a|}\kos{X}{a}\lto\,.
\]
Thus, when $\kos{X}{a^{n}}$ is in $\Thick_{\sfT}(\kos Xa)$, so is $\kos{X}{a^{n+1}}$. This justifies the claim.

It follows from Proposition~\ref{prop:hocolim-local} that $\gam_{\mcV(a)}X$ is the homotopy colimit of objects $\Si^{-1}\kos X{a^{n}}$, and hence in $\Loc_{\sfT}(\kos Xa)$, by the claim above.

Now suppose $\fa=(a_1,\ldots,a_n)$, and set $\fa'=(a_{1},\dots,a_{n-1})$.  Then the equality $\mcV(\fa)=\mcV(a_1)\cap\mcV(\fa')$ yields $\gam_{\mcV(\fa)}=\gam_{\mcV(a_n)} \gam_{\mcV(\fa')}$ by \cite[Proposition~6.1]{\bik:2008a}.  By induction on $n$ one may assume that $\gam_{\mcV(\fa')}X$ is in $\Loc_{\sfT}(\kos X{\fa'})$. Therefore $\gam_{\mcV(\fa)}X$ is in $\Loc_{\sfT}(\gam_{\mcV(a_{n})}(\kos X{\fa'}))$. The basis of the induction implies that $\gam_{\mcV(a_{n})}(\kos X{\fa'})$ is in the localizing subcategory generated by $\kos{(\kos X{\fa'})}a_{n}$, that
is to say, by $\kos X{\fa}$. Therefore, $\gam_{\mcV(\fa)}X$ is in $\Loc_{\sfT}(\kos X{\fa})$, as claimed.

(2) is an immediate consequence of (1).

(3) Since $\kos X{\fa}$ is in $\Thick_{\sfT}(X)$, it follows from (1)
that $\gam_{\mcV(\fa)}X$ is in $\Loc_{\sfT}(X)$. The localization
triangle \eqref{eq:locseq} then yields that $\bloc_{\mcV(\fa)}X$ is
also in $\Loc_{\sfT}(X)$.
\end{proof}

\section{A local-global principle}
\label{sec:A local-global principle}
We introduce a local-global principle for $\sfT$ and explain how, when
it holds, the problem of classifying the localizing subcategories can
be reduced to one of classifying localizing subcategories supported at
a single point in $\Spec R$.

Recall that $\sfT$ is a compactly generated $R$-linear triangulated
category. If for each object $X$ in $\sfT$ there is an equality
\[
\Loc_\sfT(X)=\Loc_{\sfT}(\{\gam_{\fp}X\mid \fp\in\Spec R\})
\]
we say that the \emph{local-global principle holds} for $\sfT$.

\begin{theorem}
\label{thm:localglobal}
Let $\sfT$ be a compactly generated $R$-linear triangulated
category. The local-global principle is equivalent to each of the
following statements.
\begin{enumerate}[{\quad\rm(1)}]
\item For any $X\in\sfT$ and any localizing subcategory $\sfS$ of
  $\sfT$, one has
\[
X\in \sfS \iff \gam_{\fp}X\in \sfS \text{ for each }\fp\in\Spec R\,.
\]
\item For any $X\in\sfT$, one has $X\in \Loc_{\sfT}(\{\gam_{\fp}X\mid
  \fp\in\Spec R\})$.
\item For any $X\in\sfT$ and any specialization closed subset $\mcV$
  of $\Spec R$, one has
\[
\gam_{\mcV}X\in \Loc_{\sfT}(\{\gam_{\fp}X\mid \fp\in\mcV\})\,.
\]
\item For any $X\in\sfT$, one has
  $\Loc_\sfT(X)=\Loc_{\sfT}(\{X_{\fp}\mid \fp\in\Spec R\})$.
\item For any $X\in\sfT$ and any localizing subcategory $\sfS$ of
  $\sfT$, one has
\[
X\in \sfS \iff X_{\fp}\in \sfS \text{ for each }\fp\in\Spec R\,.
\]
\item For any $X\in\sfT$, one has $X\in \Loc_{\sfT}(\{X_{\fp}\mid
  \fp\in\Spec R\})$.
\end{enumerate}
\end{theorem}

{}The proof uses some results which may also be useful elsewhere.

\begin{lemma}
\label{lem:lchlocal}
Let $X$ be an object in $\sfT$. Suppose that for any specialization
closed subset $\mcV$ of $\Spec R$, one has
\[
\gam_{\mcV}X\in \Loc_{\sfT}(\{\gam_{\fp}X\mid \fp\in\mcV\})\,.
\]
Then $\gam_{\mcV}X$ and $\bloc_\mcV X$ belong to $\Loc_{\sfT}(X)$ for
every specialization closed $\mcV\subseteq \Spec R$.
\end{lemma}

\begin{proof}
It suffices to prove that $\gam_{\fp}X$ is in $\Loc_{\sfT}(X)$ for
each $\fp$, that is to say, that the set $\mcU =\{\fp\in\Spec
R\mid\gam_{\fp}X\not\in \Loc_{\sfT}(X)\}$ is empty. Assume $\mcU$ is
not empty and choose a maximal element, say $\fp$, with respect to
inclusion. This is possible since $R$ is noetherian. Set
$\mcW=\mcV(\fp)\setminus\{\fp\}$, and consider the localization
triangle
\[
\gam_{\mcW}X \lto \gam_{\mcV(\fp)}X\lto \gam_{\fp}X\lto
\] 
of $\gam_{\mcV(\fp)}X$ with respect to $\mcW$. The hypothesis implies
the first inclusion below
\[
\gam_{\mcW}X\in \Loc_{\sfT}(\{\gam_{\fq}X\mid \fq\in\mcW\})\subseteq
\Loc_{\sfT}(X)\,,
\]
and the second one follows from the choice of $\fp$. The object
$\gam_{\mcV(\fp)}X$ is in $\Loc_{\sfT}(X)$, by
Proposition~\ref{prop:kosproperties}, so the exact triangle above
yields that $\gam_{\fp}X$ is in $\Loc_{\sfT}(X)$. This contradicts
the choice of $\fp$, and hence $\mcU=\emptyset$, as desired.
\end{proof}

\subsection*{Finite dimension}
The \emph{dimension} of a subset $\mcU$ of $\Spec R$, denoted
$\dim\mcU$, is the supremum of all integers $n$ such that there exists
a chain $\fp_0\subsetneq\fp_1\subsetneq\cdots \subsetneq\fp_n$ in
$\mcU$.  The set $\mcU$ is called \emph{discrete} if $\dim\mcU=0$.

\begin{proposition}
\label{prop:discrete}
Let $X$ be an object of $\sfT$ and set $\mcU=\supp_RX$. If $\mcU$ is
discrete, then there are natural isomorphisms
\[
X \stackrel{\sim}\longleftarrow \coprod_{\fp\in\mcU}\gam_{\mcV(\fp)}
X\stackrel{\sim}\longrightarrow \coprod_{\fp\in\mcU}\gam_\fp X\,.
\]  
\end{proposition}
\begin{proof}
Arguing as in the proof of \cite[Theorem~7.1]{\bik:2008a} one gets
that the morphisms $\gam_{\mcV(\fp)}X\to X$ induce the isomorphism on
the left, in the statement above.  The isomorphism on the right holds
since for each $\fp\in\mcU$ the morphism $\gam_{\mcV(\fp)}X\to
\gam_\fp X$ is an isomorphism by Lemma~\ref{lem:gamp}.
\end{proof}

\begin{theorem}
\label{thm:findim}
Let $\sfT$ be a compactly generated $R$-linear triangulated category
and $X$ an object of $\sfT$. If $\dim\supp_R X<\infty$, then $X$
is in $\Loc_\sfT(\{\gam_{\fp}X\mid \fp\in\supp_RX\})$.
\end{theorem}

\begin{proof}
Set $\mcU=\supp_RX$ and $\sfS=\Loc_\sfT(\{\gam_{\fp}X\mid
\fp\in\mcU\})$. The proof is an induction on $n=\dim\mcU$.  The case
$n=0$ is covered by Proposition~\ref{prop:discrete}. For $n>0$ set
$\mcU'=\mcU\setminus\min\mcU$, where $\min\mcU$ is the set of minimal
elements with respect to inclusion in $\mcU$, and set $\mcV=\cl\mcU'$.
It follows from Lemma~\ref{lem:gamp} that $\supp_R\gam_\mcV
X=\mcU'$. Since $\dim\mcU'=\dim\mcU - 1$, the induction hypothesis
yields that $\gam_\mcV X$ is in $\sfS$. On the other hand, $\supp_R
L_\mcV X=\min\mcU$ is discrete and therefore $L_\mcV X$ belongs to
$\sfS$ by Proposition~\ref{prop:discrete} and
Lemma~\ref{lem:gamp}. Thus $X$ is in $\sfS$, in view of the
localization triangle~\eqref{eq:locseq}.
\end{proof}

\subsection*{Proof of Theorem~\ref{thm:localglobal}}
It is easy to check that the local-global principle is equivalent to
(1). Also, the implications (1)$\implies$(2) and
(4)$\iff$(5)$\implies$(6) are obvious.

(2)$\implies$(3): Fix $X\in\sfT$ and a specialization closed subset
   $\mcV$ of $\Spec R$.  Then
\[
\gam_\mcV X\in \Loc_{\sfT}(\{\gam_{\fp}\gam_\mcV X\mid \fp\in\Spec
R\})= \Loc_{\sfT}(\{\gam_{\fp} X\mid \fp\in\mcV\})\,
\]
hold, where the last equality follows from Lemma~\ref{lem:gamp}.

(3)$\implies$(1): Since
   $\gam_{\fp}=\gam_{\mcV(\fp)}\bloc_{\mcZ(\fp)}$, it follows from
   condition (3) and Lemma~\ref{lem:lchlocal} that $\gam_{\fp}X$ is in
   $\Loc_{\sfT}(X)$. This implies $\Loc_{\sfT}(X)\supseteq
   \Loc_{\sfT}(\{\gam_{\fp}X\mid\fp\in\Spec R\})$ and the reverse
   inclusion holds by condition (3) for $\mcV=\Spec R$. Thus the
   local-global principle, which is equivalent to condition (1),
   holds.

(3)$\implies$(4): We have $\gam_\fp
X=\gam_{\mcV(\fp)}X_\fp\in\Loc_\sfT(X_\fp)$ for each prime ideal $\fp$
by Proposition~\ref{prop:kosproperties} and hence the hypothesis
implies $X\in \Loc_{\sfT}(\{X_{\fp}\mid \fp\in\Spec R\})$. On the
other hand, $X_\fp\in\Loc_\sfT(X)$ for each prime ideal $\fp$ by
Lemma~\ref{lem:lchlocal}.

(6)$\implies$(2): Fix $X\in\sfT$. For every prime ideal $\fp$, one
   has, for example from Lemma~\ref{lem:gamp}, that $\supp_R X_\fp$ is
   a subset of $\{\fq\in\Spec R\mid \fq\subseteq \fp\}$. In
   particular, it is finite dimensional, since $R$ is noetherian, so
   $X_\fp\in \Loc_{\sfT}(\{\gam_{\fq}X_\fp\mid \fq\in\Spec R\})$
   holds, by Theorem~\ref{thm:findim}. Thus
\begin{align*} 
X\in\Loc_{\sfT}(\{X_{\fp}\mid \fp\in\Spec R\}) &\subseteq
   \Loc_{\sfT}(\{\gam_{\fq}X_\fp\mid \fp,\fq\in\Spec R\})\\ &=
   \Loc_{\sfT}(\{\gam_{\fq}X\mid \fq\in\Spec R\}),
\end{align*}
where the last equality follows from Lemma~\ref{lem:gamp}. \qed

\medskip

The result below is an immediate consequence of Theorems~\ref{thm:findim} and \ref{thm:localglobal}.

\begin{corollary}
\label{cor:findim}
When $\dim\Spec R$ is finite the local-global principle holds for
$\sfT$. \qed
\end{corollary}

\subsection*{Classifying localizing subcategories}
Localizing subcategories of $\sfT$ are related to subsets of
$\mcV=\supp_{R}\sfT$ via the following maps
\[
\left\{
\begin{gathered}
\text{Localizing}\\ \text{subcategories of $\sfT$}
\end{gathered}\; \right\} 
\xymatrix@C=3pc {\ar@<1ex>[r]^-{{\sigma}} & \ar@<1ex>[l]^-{{\tau}}}
\left\{
\begin{gathered}
  \text{Families $(\sfS(\fp))_{\fp\in\mcV}$ with $\sfS(\fp)$ a}\\
  \text{localizing subcategory of $\gam_{\fp}\sfT$}
\end{gathered}\;
\right\}
\] 
which are defined by ${\sigma}(\sfS)=(\sfS\cap\gam_{\fp}\sfT)_{\fp\in\mcV}$ and
${\tau}(\sfS(\fp))_{\fp\in\mcV}=\Loc_{\sfT}\big(\sfS(\fp)\mid \fp\in \mcV\big)$.

The next result expresses the local-global principle in terms of these maps.

\begin{proposition}
\label{prop:classification}
The following conditions are equivalent.
\begin{enumerate}[\quad\rm(1)]
\item The local-global principle holds for $\sfT$.
\item The map ${\sigma}$ is bijective, with inverse ${\tau}$.
\item The map ${\sigma}$ is one-to-one.
\end{enumerate}
\end{proposition}

\begin{proof}
We repeatedly use the fact that $\gam_{\fp}$ is an exact functor preserving coproducts. For each localizing subcategory $\sfS$ of $\sfT$ and each $\fp$ in $\Spec R$ there is an inclusion
\begin{equation}
\label{eq:local}
\sfS \cap\gam_{\fp}\sfT \subseteq \gam_{\fp}\sfS\,.
\end{equation}

We claim that ${\sigma}{\tau}$ is the identity, that is to say, that for any family $(\sfS(\fp))_{\fp\in\mcV}$ of localizing subcategories with $S(\fp)\subseteq \gam_{\fp}\sfT$ the localizing subcategory generated by all the $S(\fp)$, call it $\sfS$, satisfies
\[
\sfS \cap \gam_{\fp}\sfT=\sfS(\fp)\,,\quad\text{for each  $\fp\in\mcV$.}
\]
To see this, note that $\gam_{\fp}\sfS = \sfS(\fp)$ holds, since $\gam_{\fp}\gam_{\fq}=0$ when $\fp\ne\fq$. Hence \eqref{eq:local} yields an inclusion $\sfS \cap \gam_{\fp}\sfT \subseteq \sfS(\fp)$. The reverse inclusion is obvious.

(1)$\implies$(2): It suffices to show that $\tau\sigma$ equals the identity, since $\sigma\tau=\id$ holds. Fix a localizing subcategory $\sfS$ of $\sfT$. It is clear that $\tau\sigma(\sfS)\subseteq \sfS$. As to the reverse inclusion: Fixing $X$ in $\sfS$, it follows from Theorem~\ref{thm:localglobal}(1) that $\gam_{\fp}X$ is in $\sfS\cap \gam_{\fp}\sfT$ and hence in $\tau\sigma(\sfS)$, for each $\fp\in\Spec R$. Thus, 
$X$ is in $\tau\sigma(\sfS)$, again by Theorem~\ref{thm:localglobal}(1).

(2)$\implies$(3): Clear.

(3)$\implies$(1): Since $\sigma\tau=\id$ and $\sigma$ is one-to-one, one gets $\tau\sigma=\id$.  For each object $X$ in $\sfT$ there is thus an equality:
\begin{align*}
\Loc_{\sfT}(X) &= \Loc_{\sfT}(\{\Loc_{\sfT}(X)\cap\gam_{\fp}\sfT\mid\fp\in\Spec R\})\\
&\subseteq \Loc_{\sfT}(\{\gam_{\fp}X \mid\fp\in\Spec R\})
\end{align*}
The inclusion follows from \eqref{eq:local}. Now apply Theorem~\ref{thm:localglobal}.
\end{proof}

The local-global principle focuses attention on the subcategory
$\gam_{\fp}\sfT$. Next we describe some of its properties, even though
these are not needed in the sequel.

\subsection*{Local structure}
Let $\fp$ be a point in $\Spec R$. In analogy with the case of $R$-modules,
we say that an object $X$ in $\sfT$ is \emph{$\fp$-local} if 
\[
\supp_R X\subseteq\{\fq\in\Spec R\mid\fq\subseteq\fp\}
\]
and that $X$ is \emph{$\fp$-torsion} if 
\[
\supp_R X\subseteq\{\fq\in\Spec R\mid\fq\supseteq\fp\}.
\]
The objects of $\gam_{\fp}\sfT$ are precisely those that are both
$\fp$-local and $\fp$-torsion; see \cite[Corollary 5.9]{\bik:2008a}
for alternative descriptions. Set
\[
X(\fp) = (\kos X\fp)_{\fp}\,.
\]
The subcategory $\Loc_{\sfT}(X(\fp))$ is independent of the choice of a generating set for the ideal $\fp$ used to construct $\kos X\fp$; this follows from the result below.

\begin{lemma}
\label{lem:residues}
The following statements hold for each $X\in\sfT$ and $\fp\in\Spec R$.
\begin{enumerate}[\quad\rm(1)]
\item $X(\fp)$ is $\fp$-local and $\fp$-torsion.
\item $\Loc_{\sfT}(X(\fp))=\Loc_{\sfT}(\gam_{\fp}X)$.
\item $\Hom_{\sfT}(W,X(\fp))=0$ for any object $W$ that is $\fq$-local and
$\fq$-torsion with $\fq\ne\fp$.
\end{enumerate}
\end{lemma}

\begin{proof}
The argument is based on the fact that the localization functor that
takes an object $X$ to $X_{\fp}$ is exact and preserves coproducts.

(1) Exactness of localization implies $(\kos X\fp)_{\fp}$ can be
realized as $\kos {X_\fp}\fp$.  Hence $X(\fp)$ belongs to
$\Thick_\sfT(X_\fp)$, so that it is $\fp$-local; it is $\fp$-torsion
by \cite[Lemma~5.11]{\bik:2008a}.

(2) Applying the localization functor to the equality
$\Loc_{\sfT}(\kos X{\fp})= \Loc_{\sfT}(\gam_{\mcV({\fp})}X)$ in
Proposition~\ref{prop:kosproperties} yields (2).

(3) If $\fq\not\subseteq\fp$ holds, then $\gam_{\mcV(\fq)}(X(\fp))=0$
and hence the desired claim follows from the adjunction isomorphism
$\Hom_{\sfT}(W,X(\fp)) \cong \Hom_{\sfT}(W,\gam_{\mcV(\fq)}X(\fp))$.
If $\fq\subseteq\fp$, then the $R$-module $\Hom_{\sfT}^{*}(W,X(\fp))$
is $\fq$-local, by Proposition~\ref{prop:cohlocal}, and $\fp$-torsion,
by \cite[Lemma 5.11]{\bik:2008a}, and hence zero since  $\fq\ne\fp$.
\end{proof}

\begin{proposition}
\label{prop:catgammap}
For each $\fp$ in $\Spec R$ and each compact object $C$ in $\sfT$, the
object $C(\fp)$ is compact in $\gam_{\fp}\sfT$, and both $\{C(\fp)
\mid C\in \sfT^{\sfc}\}$ and $\{\gam_{\fp}C \mid C\in \sfT^{\sfc}\}$
generate the triangulated category $\gam_{\fp}\sfT$. Furthermore, the
$R$-linear structure on $\sfT$ induces a natural structure of an
$R_{\fp}$-linear triangulated category on $\gam_{\fp}\sfT$.
\end{proposition}

\begin{proof}
Recall that $\gam_{\fp}T$ is a localizing subcategory of $\sfT$, so the coproduct in it is the same as the one in $\sfT$. Each object $X$ in $\gam_{\fp}\sfT$ is $\fp$-local, so there is an isomorphism
\[
\Hom_{\sfT}(C(\fp),X)\cong \Hom_{\sfT}(\kos C{\fp},X)\,.
\]
When $C$ is compact in $\sfT$, so is $\kos C{\fp}$. Thus the isomorphism above implies that $C(\fp)$ is compact in $\gam_{\fp}\sfT$. Furthermore, the collection of objects $\kos C{\fp}$ with $C$ compact in $\sfT$ generates $\sfT_{\mcV(\fp)}$ by Proposition~\ref{prop:comp}, and hence the $C(\fp)$ generate $\gam_{\fp}\sfT$.

The class of compact objects $C$ generates $\sfT$ hence the objects $\gam_\fp C$ generate $\gam_\fp\sfT$.

Proposition~\ref{prop:cohlocal} implies that for each pair of objects $X,Y$ in $\gam_\fp\sfT$ the $R$-module $\Hom^*_\sfT(X,Y)$ is $\fp$-local, so that they admit a natural $R_{\fp}$-module structure. This translates to an action of $R_{\fp}$ on $\gam_{\fp}\sfT$.
\end{proof}

\section{Stratification}
\label{sec:Stratification}

In this section we introduce a notion of stratification for
triangulated categories with ring actions. It is based on the concept
of a minimal subcategory introduced by Hovey, Palmieri, and
Strickland~\cite[\S6]{Hovey/Palmieri/Strickland:1997a}.

As before $\sfT$ is a compactly generated $R$-linear triangulated
category.

\subsection*{Minimal subcategories}
A localizing subcategory of $\sfT$ is said to be \emph{minimal} if it
is nonzero and has no proper nonzero localizing subcategories.

\begin{lemma}
\label{lem:loc-minimal}
A nonzero localizing subcategory $\sfS$ of $\sfT$ is minimal if and only if for all nonzero objects $X,Y$ in $\sfS$ one has $\Hom^*_{\sfT}(X,Y)\ne 0$.
\end{lemma}

\begin{proof}
When $\sfS$ is minimal and $X$ a nonzero object in it
$\Loc_{\sfT}(X)=\sfS$, by minimality, so if $\Hom_{\sfT}^{*}(X,Y)=0$
for some $Y$ in $\sfS$, then $\Hom_{\sfT}^{*}(Y,Y)=0$, that is to say,
$Y=0$.

Suppose $\sfS$ contains a nonzero proper localizing subcategory
$\sfS'$; we may assume $\sfS'=\Loc_{\sfT}(X)$ for some nonzero object
$X$. For each object $W$ in $\sfT$ there is then an exact triangle
$W'\xra{\theta} W\xra{\eta}W''\to$ with $W'\in\sfS'$,
$\Hom_\sfT^*(X,W'')=0$, and $\theta$ invertible if and only if and $W$
is in $\sfS'$; see Lemma~\ref{lem:loc-set}. It remains to pick an
object $W$ in $\sfS\setminus\sfS'$, set $Y=W''$, and note that $Y$ is
in $\sfS$ and nonzero.
\end{proof}

\subsection*{Stratification}
We say that $\sfT$ is \emph{stratified by $R$} if the following
conditions hold:
\begin{enumerate}[{\quad\rm(S1)}]
\item The local-global principle, discussed in Section~\ref{sec:A
local-global principle}, holds for $\sfT$.
\item For each $\fp\in \Spec R$ the localizing subcategory
$\gam_{\fp}\sfT$ is either zero or minimal.
\end{enumerate}
The crucial condition here is (S2); for example, (S1) holds when the
dimension of $\Spec R$ is finite, by Corollary~\ref{cor:findim}. Since
the objects in $\gam_\fp\sfT$ are precisely the $\fp$-local and
$\fp$-torsion ones in $\sfT$, condition (S2) is that each nonzero
$\fp$-local $\fp$-torsion object builds every other such object.

Given a localizing subcategory $\sfS$ of $\sfT$ and a subset $\mcU$ of
$\Spec R$ set
\[
\supp_R\sfS = \bigcup_{X\in\sfS}\supp_{R}X \quad\text{and}\quad
\supp_R^{-1}\mcU = \{X\in\sfT\mid \supp_RX\subseteq\mcU\}.
\]
Observe that $\supp_R$ and $\supp_R^{-1}$ both preserve inclusions.

\begin{theorem}
\label{thm:classifying}
Let $\sfT$ be a compactly generated $R$-linear triangulated
category. If $\sfT$ is stratified by $R$, then there are inclusion
preserving inverse bijections:
\[
\left\{
\begin{gathered}
\text{Localizing}\\ \text{subcategories of $\sfT$}
\end{gathered}\;
\right\} \xymatrix@C=3pc{ \ar@<1ex>[r]^-{\supp_R} &
  \ar@<1ex>[l]^-{\supp_R^{-1}}} \Big\{
\begin{gathered}
  \text{Subsets of $\supp_{R} \sfT$}
\end{gathered}\;
\Big\}
\]
Conversely, if the map $\supp_R$ is injective, then $\sfT$ must be
stratified by $R$.
\end{theorem}

\begin{proof}
For each $\fp\in\Spec R$ the subcategory $\Ker\gam_{\fp}$ is
localizing. This implies that for any subset $\mcU$ of $\Spec R$ the
subcategory $\supp_R^{-1}\mcU$ is localizing, for
\[
\supp_R^{-1}\mcU=\bigcap_{\fp\not\in\mcU}\Ker\gam_\fp\,.
\]
Moreover, it is clear that $\supp_R(\supp_R^{-1}\mcU)=\mcU$ for each subset $\mcU$ of $\supp_R\sfT$, and that $ \sfS\subseteq \supp_R^{-1}(\supp_R\sfS)$ holds for any localizing subcategory $\sfS$. The moot point is whether $\sfS$ contains $\supp_R^{-1}(\supp_R\sfS)$; equivalently, whether $\supp_R$ is one-to-one.

The map $\supp_{R}$ factors as $\sigma'\sigma$ with $\sigma$ as in Proposition~\ref{prop:classification} and $\sigma'$ the map
\[
\left\{
\begin{gathered}
  \text{Families $(\sfS(\fp))_{\fp\in\supp_{R}\sfT}$ with
  $\sfS(\fp)$}\\ \text{a localizing subcategory of $\gam_{\fp}\sfT$}
\end{gathered}\;
\right\} \lto \Big\{
\begin{gathered}
  \text{Subsets of $\supp_{R} \sfT$}
\end{gathered}\;
\Big\}
\]
where $\sigma'(\sfS(\fp))=\{\fp\in\Spec R\mid
S(\fp)\ne\{0\}\}$. Evidently $\sigma'$ is one-to-one if and only if it
is bijective, if and only if the minimality condition (S2) holds. The
map $\sigma$ is also one-to-one if and only if it is bijective;
moreover this holds precisely when the local-global principle holds for
$\sfT$, by Proposition~\ref{prop:classification}. The desired result
follows.
\end{proof}

\begin{corollary}
If $R$ stratifies $\sfT$ and $\sfG$ is a set of generators for $\sfT$,
then each localizing subcategory $\sfS$ of $\sfT$ is generated by the
set $\sfS\cap\{\gam_{\fp}X\mid X\in \sfG,\, \fp\in\Spec R\}$.  In
particular, there exists a localization functor $L\col\sfT\to\sfT$
such that $\sfS=\Ker L$.
\end{corollary}

\begin{proof}
The first assertion is an immediate consequence of
Theorem~\ref{thm:classifying}, since $\sfS$ and the localizing
subcategory generated by the given set of objects have the same
support.  Given this, the second one follows from
Lemma~\ref{lem:loc-set}.
\end{proof}

Other consequences of stratification are given in
Sections~\ref{sec:Orthogonality} and \ref{sec:Classifying thick
subcategories}. Now we provide examples of triangulated categories
that are stratified; see also Example~\ref{eg:kg}.

\begin{example}
\label{ex:derivedcategory}
Let $A$ be a commutative noetherian ring and $\sfD(A)$ the derived
category of the category of $A$-modules. The category $\sfD(A)$ is
compactly generated, $A$-linear, and triangulated. This example is
discussed in \cite[\S8]{\bik:2008a}, where it is proved that the
notion of support introduced in \cite{\bik:2008a} coincides with the
usual one, due to Foxby and Neeman; see \cite[Theorem
9.1]{\bik:2008a}. In view of Theorem~\ref{thm:classifying}, one can
reformulate \cite[Theorem 2.8]{Neeman:1992a} as: The $A$-linear
triangulated category $\sfD(A)$ is stratified by $A$.  This example
will be subsumed in Theorem~\ref{thm:cdga}.
\end{example}

\begin{example}
Let $k$ be a field and $\Lambda$ an exterior algebra over $k$ in finitely many indeterminates of negative odd degree; the grading is upper. We view $\Lambda$ as a dg algebra, with differential zero. In \cite[\S6]{\bik:2008b} we introduced the homotopy category of graded-injective dg $\Lambda$-modules and proved that it is stratified by a natural action of its cohomology algebra, $\Ext^{*}_{\Lambda}(k,k)$.
\end{example}

The next example shows that there are triangulated categories which cannot be stratified by any ring action.

\begin{example}
\label{ex:quiver}
Let $k$ be a field and $Q$ a quiver of Dynkin type; see, for example, \cite[Chapter 4]{Benson:1991a}. The path algebra $kQ$ is a finite dimensional hereditary algebra of finite representation type. It is easily checked that the graded center of the derived category $\sfD(kQ)$ is isomorphic to $k$. In fact, each object in $\sfD(kQ)$ is a direct sum of indecomposable objects, and $\End^*_{\sfD(kQ)}(X)\cong k$ for each indecomposable object $X$.  The localizing subcategories of $\sfD(kQ)$ are parameterized by the noncrossing partitions associated to $Q$; this can be deduced from work of Ingalls and Thomas \cite{Ingalls/Thomas:2009}. Thus the triangulated category $\sfD(kQ)$ is stratified by some ring acting on it if and only if the quiver consists of one vertex and has no arrows.
\end{example}

\section{Orthogonality}
\label{sec:Orthogonality}
Let $X$ and $Y$ be objects in $\sfT$. The discussion below is motivated by the question: when is $\Hom^{*}_{\sfT}(X,Y)=0$? The orthogonality property \cite[Corollary 5.8]{\bik:2008a} says that if $\cl(\supp_{R}X)$ and $\supp_{R}Y$ are disjoint, then one has the vanishing. What we seek are converses to this statement, ideally in terms of the supports of $X$ and $Y$. Lemma~\ref{lem:loc-minimal} suggests that one can expect satisfactory answers only when $\sfT$ is stratified. In this section we establish some results addressing this question and give examples which indicate that these may be the best possible.

For any graded $R$-module $M$ set $\Supp_R M=\{\fp\in\Spec R\mid M_\fp\ne 0\}$. This subset is sometimes referred to as the `big support' of $M$ to distinguish it from its `homological' support, $\supp_{R}M$. Analogously, for any object $X$ in $\sfT$, we set
\[
\Supp_R X=\bigcup_{C\in\sfT^c}\Supp_R\Hom^*_\sfT(C,X)\,.
\]
It follows from \cite[Theorem 5.15(1) and Lemma~2.2(1)]{\bik:2008a}
that there is an equality:
\[
\Supp_{R}X=\cl(\supp_R X)\,.
\]
We use this equality without further comment.

\begin{theorem}
\label{thm:intersection}
Let $\sfT$ be a compactly generated $R$-linear triangulated
category. If $R$ stratifies $\sfT$, then for each compact object $C$
and each object $Y$, there is an equality
\[
\Supp_{R} \Hom^{*}_{\sfT}(C,Y) = \Supp_{R}C \cap \Supp_{R}Y\,.
\]
\end{theorem}

The proof requires only stratification condition (S2), never (S1).

\begin{proof}
Fix a prime ideal $\fp\in\Spec R$.  Suppose
$\Hom^{*}_{\sfT}(C,Y)_{\fp}\ne 0$; by definition, one then has
$\fp\in\Supp_R Y$. Moreover $\End^{*}_{\sfT}(C)_{\fp}\ne 0$ since the
$R$-action on $\Hom^{*}_{\sfT}(C,Y)_{\fp}$ factors through it, hence
$\fp$ is also in $\Supp_{R}C$. Thus there is an inclusion
\[
\Supp_{R} \Hom^{*}_{\sfT}(C,Y) \subseteq \Supp_{R} C \cap
\Supp_{R}Y\,.
\]

Now suppose $\Hom^{*}_{\sfT}(C,Y)_{\fp}=0$. One has to verify that
that for any prime ideal $\fq\subseteq\fp$ either $\gam_{\fq}C=0$ or
$\gam_{\fq}Y=0$. By \cite[Theorem~4.7]{\bik:2008a}, see also
Proposition~\ref{prop:cohlocal}, since $C$ is compact the adjunction
morphism $Y\to Y_{\fq}$ induces an isomorphism
\[
0=\Hom^{*}_{\sfT}(C,Y)_{\fq} \cong \Hom^{*}_{\sfT}(C,Y_{\fq})\,.
\]
As $\gam_{\mcV(\fq)}Y$ is in $\Loc_{\sfT}(Y)$, by
Proposition~\ref{prop:kosproperties}, one obtains that $\gam_{\fq}Y$
is in $\Loc_{\sfT}(Y_{\fq})$, hence the calculation above yields
$\Hom^{*}_{\sfT}(C,\gam_{\fq}Y)=0$. As $\gam_{\fq}Y$ is $\fq$-local
the adjunction morphism $C\to C_{\fq}$ induces the isomorphism below
\[
\Hom^{*}_{\sfT}(C_{\fq},\gam_{\fq}Y)\cong
\Hom^{*}_{\sfT}(C,\gam_{\fq}Y)=0\,.
\]
Using now the fact that $\gam_{\fq}C$ is in $\Loc_\sfT(C_{\fq})$ one
gets $\Hom_\sfT^{*}(\gam_{\fq}C,\gam_{\fq}Y)=0$. Our hypothesis was
that $R$ stratifies $\sfT$. Thus one of $\gam_{\fq}C$ or $\gam_{\fq}Y$
is zero.
\end{proof}

The example below shows that the conclusion of the preceding theorem
need not hold when $C$ is not compact.  See also
Example~\ref{ex:orthogonality}

\begin{example}
\label{ex:intersection}
Let $A$ be a commutative noetherian ring with Krull dimension at least
one and $\fm$ a maximal ideal of $A$ that is not also a minimal
prime. For example, take $A=\bbZ$ and $\fm=(p)$, where $p$ is a prime
number.

Let $\sfT$ be the derived category of $A$-modules, viewed as an
$A$-linear category; see Example~\ref{ex:derivedcategory}. Let $E$ be
the injective hull of $A/\fm$. The $A$-module $\Hom^{*}_{\sfT}(E,E)$
is then the $\fm$-adic completion of $A$, so it follows that
\[
\Supp_{A}\Hom^{*}_\sfT(E,E) = \{\fp\subseteq \fm\mid \fp\in \Spec
R\}\supsetneq \{\fm\} = \Supp_{A}E\,.
\]
Observe that $\supp_{A}\Hom^{*}_\sfT(E,E)=
\Supp_{A}\Hom^{*}_\sfT(E,E)$ and $\supp_{A}E=\Supp_{A}E$.
\end{example}

One drawback of Theorem~\ref{thm:intersection} is that it involves the
big support $\Supp_{R}$, while one is mainly interested in
$\supp_{R}$. Next we identify a rather natural condition on $\sfT$
under which one can obtain results in the desired form.

\subsection*{Noetherian categories}
We call a compactly generated $R$-linear triangulated category
\emph{noetherian} if for any compact object $C$ in $\sfT$ the
$R$-module $\End^{*}_{\sfT}(C)$ is finitely generated. This is
equivalent to the condition that for all compact objects $C,D$ the
$R$-module $\Hom_{\sfT}^{*}(C,D)$ is finitely generated: consider
$\End^{*}_{\sfT}(C\oplus D)$. If $C$ generates $\sfT$, then $\sfT$ is
noetherian if and only if the $R$-module $\End^{*}_{\sfT}(C)$ is
noetherian.

As a consequence of Theorem~\ref{thm:intersection} one gets:
 
\begin{corollary}
\label{cor:intersection}
If $\sfT$ is noetherian and stratified by $R$, then for each pair of
compact objects $C,D$ in $\sfT$ there is an equality
\[
\supp_{R}\Hom^{*}_{\sfT}(C,D) = \supp_{R}C\cap \supp_{R}D\,.
\]
When in addition $R^{i}=0$ holds for $i<0$, one has
$\Hom^{n}_{\sfT}(C,D)=0$ for $n\gg 0$ if and only if
$\Hom^{n}_{\sfT}(D,C)=0$ for $n\gg 0$.
\end{corollary}

\begin{proof}
In view of the noetherian hypothesis and \cite[Lemma 2.2(1),
Theorem~5.5(2)]{\bik:2008a}, the desired equality follows from
Theorem~\ref{thm:intersection}. It implies in particular that
\[
\supp_{R}\Hom^{*}_{\sfT}(C,D) = \supp_{R}\Hom^{*}_{\sfT}(D,C) \,.
\]
When $R^{i}=0$ holds for $i<0$ and $M$ is a noetherian $R$-module one
has $M^{n}=0$ for $n\gg0$ if and only if $\supp_{R}M\subseteq
\{\fp\in\Spec R\mid \fp\supseteq R^{\ges 1}\}$; see
\cite[Proposition 2.4]{Bergh/Iyengar/Krause/Oppermann:2009a}. The last
part of the corollary now follows from the equality above.
\end{proof}

There is a version of the preceding result where the objects $C$ and
$D$ need not be compact. This is the topic of the next theorem. As
preparation for its proof, and for later applications, we further
develop the material in \cite[Definition 4.8]{\bik:2008a}. Let $C$ be
a compact object in $\sfT$. For each injective $R$-module $I$, the
Brown representability theorem \cite{Keller:1994a,Neeman:1996} yields
an object $T_{C}(I)$ in $\sfT$ such that there is a natural
isomorphism:
\[
\Hom_{\sfT}(-,T_{C}(I)) \cong \Hom_R(\Hom^*_\sfT(C,-),I)\,.
\]
Moreover, the assignment $I\mapsto T_{C}(I)$ defines a functor
$T_{C}\col\Inj R\to\sfT$ from the category of injective $R$-modules to
$\sfT$.

\begin{proposition}
\label{prop:iobject}
Let $C$ be a compact object in $\sfT$. The functor $T_{C}\col\Inj
R\to\sfT$ preserves products. If the $R$-linear category $\sfT$ is
noetherian, each $I\in\Inj R$ satisfies:
\[
\supp_{R}T_{C}(I) = \supp_{R}C \cap \supp_{R} I =
\supp_{R}\End^{*}_{\sfT}(C) \cap \supp_{R} I\,.
\]
In particular, for each $\fp\in\Spec R$ the object $T_{C}(E(R/\fp))$
is in $\gam_{\fp}\sfT$.
\end{proposition}

\begin{proof}
It follows by construction that $T_{C}$ preserves products. For each
compact object $D$ in $\sfT$, there is an isomorphism of $R$-modules
\[
\Hom^{*}_{\sfT}(D,T_{C}(I))\cong
\Hom^{*}_{R}(\Hom^{*}_{\sfT}(C,D),I)\,.
\]
When $\sfT$ is noetherian, so that the $R$-module
$\Hom^{*}_{\sfT}(C,D)$ is finitely generated, the isomorphism above
gives the first equality below:
\begin{align*}
\supp_{R}\Hom^{*}_{\sfT}(D,T_{C}(I)) &=
\supp_{R}\Hom^{*}_{\sfT}(C,D)\cap \supp_{R}I\\ &=
\supp_{R}C\cap\supp_{R}D\cap \supp_{R}I\,.
\end{align*}
The second equality holds by
Corollary~\ref{cor:intersection}. Lemma~\ref{lem:support-test} below
then yields the first of the desired equalities; the second one holds
by \cite[Theorem 5.5(2)]{\bik:2008a}.
\end{proof}

The following lemma provides an alternative description of the support
of an object in $\sfT$. Note that $\sfT$ need not  be noetherian.

\begin{lemma}
\label{lem:support-test}
Let $X$ be an object in $\sfT$ and $\mcU$ a subset of $\supp_{R}\sfT$. If
\[
\supp_{R}\Hom^{*}_{\sfT}(C,X)=\mcU\cap \supp_{R}C
\]
holds for each compact object $C$, then $\supp_{R}X = \mcU$.
\end{lemma}

\begin{proof}
It follows from \cite[Theorem 5.2]{\bik:2008a} that
$\supp_{R}X\subseteq \mcU$.

Fix $\fp$ in $\mcU$ and choose a compact object $D$ with
$\fp$ in $\supp_{R}D$. Then $\fp$ is in $\supp_{R}(\kos D\fp)$, so the
hypothesis yields that $\fp$ is in $\supp_{R}\Hom^{*}_{\sfT}(\kos
D\fp,X)$. Hence $\fp$ belongs to $\supp_{R}X$, by \cite[Proposition
5.12]{\bik:2008a}.
\end{proof}

For a compact object $C$, the functor $\Hom^{*}_{\sfT}(C,-)$ vanishes
on $\Loc_{\sfT}(Y)$ if and only if $\Hom^{*}(C,Y)=0$. Using this
observation, it is easy to verify that the theorem below is an
extension of Corollary~\ref{cor:intersection}. Compare it also with
\cite[Corollary 5.8]{\bik:2008a}.

\begin{theorem}
\label{thm:orthogonality}
Let $\sfT$ be an $R$-linear triangulated category that is noetherian
and stratified by $R$.  For any $X$ and $Y$ in $\sfT$ the conditions
below are equivalent:
\begin{enumerate}[{\quad\rm(1)}]
\item $\Hom^{*}_{\sfT}(X,Y')=0$ for any $Y'$ in $\Loc_{\sfT}(Y)$;
\item $\cl(\supp_{R}X) \cap \supp_{R} Y =\emptyset$.
\end{enumerate}
\end{theorem}

\begin{proof}
(1)$\implies$(2): Let $\fp$ be a point in $\supp_{R}Y$ and $C$ a
compact object in $\sfT$. Proposition~\ref{prop:iobject} yields that
$T_{C}(E(R/\fp))$ is in $\gam_{\fp}\sfT$, and hence also in
$\Loc_{\sfT}(Y)$; the last assertion holds by
Theorem~\ref{thm:classifying}. This explains the equality below:
\[
\Hom^{*}_{R}(\Hom^{*}_{\sfT}(C,X),E(R/\fp))\cong
\Hom^{*}_{\sfT}(X,T_{C}(E(R/\fp)))= 0\,,
\]
while the isomorphism follows from the definition of $T_{C}$. Thus
$\Hom^{*}_{\sfT}(C,X)_{\fp}=0$.  Since $C$ was arbitrary, this means
that $\fp$ is not in $\cl(\supp_{R}X)$.

(2)$\implies$(1): One has $\supp_{R}Y'\subseteq \supp_{R}Y$ for $Y'$
   in $\Loc_{\sfT}(Y)$, since the functor $\gam_{\fp}$ is exact and
   preserves coproducts. The orthogonality property of supports,
   \cite[Corollary 5.8]{\bik:2008a} thus implies that if condition (2)
   holds, then $\Hom^{*}_{\sfT}(X,Y')=0$.
\end{proof}

Recall that the left orthogonal subcategory of $\sfS$, denoted
${}^{\perp}\sfS$, is the localizing subcategory $\{X\in \sfT\mid
\Hom^{*}_{\sfT}(X,Y)=0\text{ for all $Y\in\sfS$}\}$. As a
straightforward consequence of Theorem~\ref{thm:orthogonality} one
obtains a description of the support of the left orthogonal of a
localizing category, answering a question raised by
Rickard.\footnote{After a talk by Iyengar at the workshop `Homological
methods in group theory', MSRI 2008.}

\begin{corollary}
\label{cor:rickard}
For each localizing subcategory $\sfS$ of $\sfT$ the following
equality holds: $\supp_{R}({}^{\perp}\sfS) = \{\fp\in\supp_{R}\sfT\mid
\mcV(\fp)\cap \supp_{R}\sfS=\emptyset\}$.\qed
\end{corollary}

\begin{remark}
In the context of Theorem~\ref{thm:orthogonality}, for any compact
object $C$ one has
\[
\Hom^{*}_{\sfT}(C,Y)=0 \text{ if and only if }\supp_{R}C\cap \supp_{R}
Y =\emptyset\,.
\]
The next example shows that one cannot do away entirely with the
hypothesis that $C$ is compact; the point being that
$\Hom^{*}_{\sfT}(X,Y)=0$ does not imply that $\Hom^{*}(X,-)$ is zero
on $\Loc_{\sfT}(Y)$, unless $X$ is compact.
\end{remark}

\begin{example}
\label{ex:orthogonality}
Let $A$ be a complete local domain and $Q$ its field of fractions. For
example, take $A$ to be the completion of $\bbZ$ at a prime $p$. It
follows from a result of Jensen~\cite[Theorem 1]{Jensen:1972a} that
$\Ext^{*}_{A}(Q,A)=0$. Thus, with $\sfT$ the derived category of $A$,
one gets $\supp_{A}\Hom^{*}_{\sfT}(Q,A)=\emptyset$ while
$\supp_{A}Q\cap \supp_{A}A$ consists of the zero ideal.  Note that $Q$
is in $\Loc_{\sfT}(A)$, so there is no contradiction with
Theorem~\ref{thm:orthogonality}.
\end{example}

\section{Classifying thick subcategories}
\label{sec:Classifying thick subcategories}
In this section we prove that when $\sfT$ is noetherian and stratified
by $R$ its thick subcategories of compact objects are parameterized by
specialization closed subsets of $\supp_{R}\sfT$. As before, $R$ is a
graded-commutative noetherian ring and $\sfT$ is a compactly generated
$R$-linear triangulated category.

\subsection*{Thick subcategories}
One can deduce the next result from the classification of localizing
subcategories, Theorem~\ref{thm:classifying}, as in
\cite[\S3]{Neeman:1992a}. We give a different proof.

\begin{theorem}
\label{thm:classifying-thick}
Let $\sfT$ be a compactly generated $R$-linear triangulated category
that is noetherian and stratified by $R$. The map
\[
\left\{
\begin{gathered}
  \text{Thick subcategories}\\ \text{of $\sfT^c$}
\end{gathered}\;
\right\} \xymatrix@C=3pc{ \ar[r]^-{\supp_{R}} &{}} \left\{
\begin{gathered}
\text{Specialization closed}\\ \text{subsets of $\supp_R\sfT$}
\end{gathered}\;
\right\}
\]
is bijective. The inverse map sends a specialization closed subset
$\mcV$ of $\Spec R$ to the subcategory $\{C\in \sfT^{c}\mid
\supp_{R}C\subseteq \mcV\}$.
\end{theorem}

Observe that in the proof the injectivity of the map $\supp_{R}$
requires only that $\sfT$ satisfies the stratification condition (S2),
while the surjectivity uses only the hypothesis that $\sfT$ is
noetherian.

\begin{proof}
First we verify that $\supp_{R}\sfC$ is specialization closed for any
thick subcategory $\sfC$ of $\sfT^{c}$. For any compact object $C$ the
$R$-module $\End^*_\sfT(C)$ is finitely generated, and this implies
$\supp_R C=\supp_R\End_\sfT^*(C)$, by \cite[Theorem~5.5]{\bik:2008a}.
Thus $\supp_RC$ is a closed subset of $\Spec R$, and therefore
$\supp_R\sfC$ is specialization closed.

To verify that the map $\supp_{R}$ is surjective, let $\mcV$ be a
specialization closed subset of $\supp_R\sfT$ and set $\sfC=\{\kos
C\fp\mid C\in \sfT^c,\, \fp\in\mcV\}$. One then has that
$\Loc_\sfT(\sfC)=\sfT_\mcV$ by \cite[Theorem~6.4]{\bik:2008a}, and
therefore the following equalities hold
\[
\supp_{R}\sfC=\supp_R\sfT_\mcV=\mcV\cap\supp_R\sfT=\mcV\,.
\]

It remains to prove that $\supp_{R}$ is injective. Let $\sfC$ be a
thick subcategory of $\sfT$ and set $\sfD=\{D\in\sfT^c\mid\supp_R
D\subseteq\supp_{R}\sfC\}$. We need to show that
$\sfC=\sfD$. Evidently, an inclusion $\sfC\subseteq\sfD$ holds. To
establish the other inclusion, let $L\col\sfT\to\sfT$ be the
localization functor with $\Ker L=\Loc_\sfT(\sfC)$; see
Lemma~\ref{lem:loc-set} for its existence. Let $D$ be an object in
$\sfD$. Each object $C$ in $\sfC$ satisfies $\Hom_\sfT^*(C,LD)=0$, so
Theorem~\ref{thm:intersection} implies $\Supp_R C\cap\Supp_R
LD=\emptyset$. Hence $LD=0$, that is to say, $D$ belongs to
$\Loc_\sfT(\sfC)$.  It then follows from
\cite[Lemma~2.2]{Neeman:1992a} that $D$ is in $\sfC$.
\end{proof}

\subsection*{Smashing subcategories}
Next we prove that when $\sfT$ is stratified and noetherian, the
telescope conjecture~\cite{Ravenel:1987a} holds for $\sfT$.  In
preparation for its proof, we record an elementary observation.

\begin{lemma}
\label{lem:smash}
Let $\fp\subseteq\fq$ be prime ideals in $\Spec R$. The injective hull
$E(R/\fp)$ of $R/\fp$ is a direct summand of a product of shifted
copies of $E(R/\fq)$.
\end{lemma}

\begin{proof}
The shifted copies of $E(R/\fq)$ form a set of injective cogenerators
for the category of $\fq$-local modules.  This implies the desired
result.
\end{proof}

A subset $\mcU$ of $\Spec R$ is said to be \emph{closed under
generalization} if $\Spec R\setminus\mcU$ is specialization
closed. More explicitly: $\fq\in \mcU$ and $\fp\subseteq\fq$ imply
$\fp\in \mcU$.

\begin{theorem}
\label{thm:smash}
Let $\sfT$ be an $R$-linear triangulated category that is noetherian
and stratified by $R$.  There is then a bijection
\[
\left\{
\begin{gathered}
  \text{Localizing subcategories of $\sfT$}\\ \text{closed under all
  products}
\end{gathered}\;
\right\} \xymatrix@C=3pc{ \ar[r]^-{\supp_{R}}&{}} \left\{
\begin{gathered}
  \text{Subsets of $\supp_{R}\sfT$}\\ \text{closed under
  generalization}
\end{gathered}\;
\right\}
\]
Moreover, if $L\col\sfT\to\sfT$ is a localization functor that
preserves arbitrary coproducts, then the localizing subcategory $\Ker
L$ is generated by objects that are compact in $\sfT$.
\end{theorem}

\begin{remark}
The inverse map of $\supp_{R}$ takes a generalization closed subset
$\mcU$ of $\Spec R$ to the category of objects $X$ of $\sfT$ with
$\supp_R X\subseteq\mcU$; in other words, the category of
$L_\mcV$-local objects, where $\mcV=\Spec R\setminus\mcU$.
\end{remark}

\begin{proof}
Let $\sfS$ be a localizing subcategory of $\sfT$ that is closed under
arbitrary products.  We know from Theorem~\ref{thm:classifying} that
$\sfS$ is determined by its support $\supp_R\sfS$. Thus we need to
show that it is closed under generalization.

Fix prime ideals $\fp\subseteq\fq$ in $\supp_{R}\sfT$ and suppose that
$\fq$ is in $\supp_R\sfS$.  It follows from
Theorem~\ref{thm:classifying} that $\gam_{\fq}\sfT\subseteq \sfS$
holds. Pick a compact object $C$ such that $\supp_{R}C$ contains
$\fp$; this is possible since $\supp_{R}\sfT^{c}=\supp_{R}\sfT$. Since
$\sfT$ is noetherian, $\supp_{R}C$ is a closed subset of $\Spec R$, by
\cite[Theorem 5.5]{\bik:2008a}, and hence contains also $\fq$. Let
$E(R/\fq)$ be the injective hull of the $R$-module $R/\fq$. Since
$\sfT$ is noetherian, Proposition~\ref{prop:iobject} yields that
$T_{C}(E(R/\fq))$ is in $\gam_{\fq}\sfT$ and hence in $\sfS$.  The
functor $T_{C}$ preserves products, so Lemma~\ref{lem:smash} implies
that $T_{C}(E(R/\fp))$ is a direct summand of $T_{C}(E(R/\fq))$ and
hence it is also in $\sfS$, because the latter is a localizing
subcategory closed under products. Another application of
Proposition~\ref{prop:iobject} shows that
$\supp_{R}T_{C}(E(R/\fp))=\{\fp\}$, so that $\fp\in\supp_{R}\sfS$
holds, as desired

Next let $\mcU$ be a generalization closed subset of $\Spec R$ and set $\mcV=\Spec R\setminus\mcU$. Let $\sfS$ be the category of $L_\mcV$-local objects, so that $\supp_R \sfS=\mcU$ holds, by \cite[Corollary~5.7]{\bik:2008a}. By construction, the category $\sfS$ is triangulated and closed under arbitrary products; it is localizing because the localization functor $L_\mcV$ preserves arbitrary coproducts, by \cite[Corollary~6.5]{\bik:2008a}.

This completes the proof that $\supp_{R}$ induces the stated bijection.

Finally, let $L\col\sfT\to\sfT$ be a localization functor that
preserves arbitrary coproducts. The category of $L$-local objects,
which always is closed under products, is then also a localizing
subcategory of $\sfT$.  The first part of this proof shows that
$L\cong L_\mcV$ for some specialization closed subset $\mcV$ of $\Spec
R$, because the localization functor $L$ is determined by the category
of $L$-local objects. It remains to note that $\Ker L$, which is the
category $\sfT_\mcV$, is generated by compact objects, by
\cite[Theorem~6.4]{\bik:2008a}.
\end{proof}

\section{Tensor triangulated categories}
\label{sec:Tensor triangulated categories}
In this section we discuss special properties of triangulated
categories which hold when they have a tensor structure. The main
result here is Theorem~\ref{thm:ttlg}, which says that the
local-global principle holds for such categories, when the
action of the tensor product is also taken into account.

Let $\sfT=(\sfT,\otimes,\one)$ be a tensor triangulated category as
defined in \cite[\S8]{\bik:2008a}. In particular, $\sfT$ is a
compactly generated triangulated category endowed with a symmetric
monoidal structure; $\otimes$ is its tensor product and $\one$ the
unit of the tensor product. It is assumed that $\otimes$ is exact in
each variable, preserves coproducts, and that $\one$ is compact.

The symmetric monoidal structure ensures that the endomorphism ring
$\End^{*}_{\sfT}(\one)$ is graded commutative. This ring acts on
$\sfT$ via homomorphisms
\[
\End^{*}_{\sfT}(\one)\xra{\ X\otimes-\ }\End^{*}_{\sfT}(X)\,,
\]
In particular, any homomorphism $R\to \End^{*}_{\sfT}(\one)$ of rings
with $R$ graded commutative induces an action of $R$ on $\sfT$. We say
that an $R$ action on $\sfT$ is \emph{canonical} if it arises from
such a homomorphism. In that case there are for each specialization
closed subset $\mcV$ and point $\fp$ of $\Spec R$ natural isomorphisms
\begin{equation}
\label{eq:loc-tensor}
\gam_{\mcV}X\cong X\otimes \gam_{\mcV}\one\,,\quad \bloc_{\mcV}X\cong
X\otimes \bloc_{\mcV}\one\,, \quad\text{and}\quad \gam_{\fp}X\cong
X\otimes \gam_{\fp}\one\,.
\end{equation}
These isomorphisms are from \cite[Theorem~8.2,
Corollary~8.3]{\bik:2008a}.\footnote{For these results to hold, the
$R$ action should be canonical, for the $R$-linearity of the
adjunction isomorphism $\Hom_{\sfT}(X\otimes Y,Z)\cong
\Hom_{\sfT}(X,\fHom(Y,Z))$ is used in the arguments.}

\subsection*{Tensor ideal localizing subcategories}

A localizing subcategory $\sfS$ of $\sfT$ said to be \emph{tensor
ideal} if for each $X\in \sfT$ and $Y\in \sfS$, the object $X\otimes
Y$, hence also $Y\otimes X$, is in $\sfS$. The smallest tensor ideal
localizing subcategory containing a subcategory $\sfS$ is denoted
$\Loc^{\otimes}_{\sfT}(\sfS)$. Evidently there is always an inclusion
$\Loc_{\sfT}(\sfS)\subseteq \Loc^{\otimes}_{\sfT}(\sfS)$; equality
holds when the unit $\one$ generates $\sfT$.

The following result is proved in \cite[Theorem 3.6]{\bik:2008b} under
the additional assumption that $\sfT$ has a single compact
generator. The same argument carries over; except that, instead of
\cite[Proposition 3.5]{\bik:2008b} use Proposition~\ref{prop:comp}
above. We omit details.

\begin{theorem}
\label{thm:ttlg}
Let $\sfT$ be a tensor triangulated category with a canonical
$R$-action. For each object $X$ in $\sfT$ there is an equality
\[
\Loc_{\sfT}^{\otimes}(X) = \Loc_{\sfT}^{\otimes}\big(\gam_\fp
X\mid\fp\in\Spec R\big)\,.
\]
In particular, when $\one$ generates $\sfT$, the local global
principle holds for $\sfT$.\qed
\end{theorem}

\subsection*{Stratification}
For each $\fp$ in $\Spec R$, the localizing subcategory
$\gam_{\fp}\sfT$, consisting of $\fp$-local and $\fp$-torsion objects,
is tensor ideal; this is immediate from \eqref{eq:loc-tensor}.  We say
that $\sfT$ is \emph{stratified} by $R$ when for each $\fp$, the
category $\gam_{\fp}\sfT$ is either zero or has no proper \emph{tensor
ideal} localizing subcategories. Note the analogy with condition (S2)
in Section~\ref{sec:Stratification}; the analogue of (S1) need not be
imposed thanks to Theorem~\ref{thm:ttlg}.

There are analogues, for tensor triangulated categories, of results in
Sections~\ref{sec:Orthogonality} and \ref{sec:Classifying thick
subcategories}; the proofs are similar, see also
\cite[\S11]{\bik:2008a}.  One has in addition also the following
`tensor product theorem'.

\begin{theorem}
\label{thm:ttintersection}
Let $\sfT$ be a tensor triangulated category with a canonical
$R$-action. If $R$ stratifies $\sfT$, then for any objects $X,Y$ in
$\sfT$ there is an equality
\[
\supp_{R}(X\otimes Y) = \supp_{R} X \cap \supp_{R}Y\,.
\] 
\end{theorem}

\begin{proof}
Fix a point $\fp$ in $\Spec R$. From \ref{eq:loc-tensor} it is easy to
verify that there are isomorphisms $\gam_{\fp}(X\otimes Y)\cong
\gam_{\fp}X\otimes\gam_{\fp}Y\cong \gam_{\fp}X\otimes Y$. These will
be used without further ado. They yield an inclusion:
\[
\supp_{R}(X\otimes Y) \subseteq \supp_{R} X \cap\supp_{R}Y\,.
\]
When $\gam_{\fp}X\ne 0$ the stratification condition yields
$\gam_{\fp}\one\in\Loc^{\otimes}(\gam_{\fp}X)$, and hence also
$\gam_{\fp}Y\in\Loc^{\otimes}(\gam_{\fp}X\otimes Y)$. Thus when
$\gam_{\fp}Y\ne 0$ also holds, $\gam_{\fp}(X\otimes Y)\ne 0$ holds,
which justifies the reverse inclusion.
\end{proof}

\begin{example}
\label{eg:kg}
Let $G$ be a finite group, $k$ a field of characteristic $p$, where
$p$ divides the order of $G$, and $kG$ the group algebra. The homotopy
category of complexes of injective $kG$-modules, $\KInj {kG}$, is a
compactly generated tensor triangulated category with a canonical
action of the cohomology ring $H^{*}(G,k)$.  One of the main results
of \cite{\bik:2008b}, Theorem 9.7, is that $\KInj{kG}$ is stratified
by this action. The same is true also of the stable module category $\StMod kG$; see
\cite[Theorem 10.3]{\bik:2008b}.
\end{example}

\section{Formal differential graded algebras}
\label{sec:Formal differential graded algebras}

The goal of this section is to prove that the derived category of
differential graded (henceforth abbreviated to `dg') modules over a
formal commutative dg algebra is stratified by its cohomology algebra,
when that algebra is noetherian. This result specializes to one of
Neeman's~\cite{Neeman:1992a} concerning rings, which may be viewed as
dg algebras concentrated in degree $0$.

For basic notions concerning dg algebras and dg modules over them we
refer the reader to Mac Lane~\cite[\S6.7]{Maclane:1963a}. A
\emph{quasi-isomorphism} between dg algebras $A$ and $B$ is a morphism
$\vf\col A\to B$ of dg algebras such that $\hh\vf$ is bijective; $A$
and $B$ are \emph{quasi-isomorphic} if there is a chain of
quasi-isomorphisms linking them. The multiplication on $A$ induces one
on its cohomology, $\hh A$. We say that $A$ is \emph{formal} if it is
quasi-isomorphic to $\hh A$, viewed as a dg algebra with zero differential.

We write $\sfD(A)$ for the derived category of dg modules over a dg
algebra $A$; it is a triangulated category, generated by the compact
object $A$; see, for instance, \cite{Keller:1994a}.
 
A dg algebra $A$ is said to be \emph{commutative} if its underlying
ring is graded commutative. In this case the derived tensor product of
dg modules, denoted $\lotimes$, endows $\sfD(A)$ with a structure of a
tensor triangulated category, with unit $A$. One is thus in the
framework of Section~\ref{sec:Tensor triangulated categories}.

The next theorem generalizes \cite[Theorem 5.2]{\bik:2008b}, which deals with the case of graded algebras of the form $k[x_{1},\dots,x_{n}]$, where $k$ is a field and $x_{1},\dots,x_{n}$ are indeterminates, of even degree if the characteristic of $k$ is not $2$.

\begin{theorem}
\label{thm:cdga}
Let $A$ be a commutative dg algebra such that the ring $\hh A$ is
noetherian. If $A$ is formal, then $\sfD(A)$ is stratified by the
canonical $\hh A$-action.
\end{theorem}

In the proof we use a totalization functor from complexes over a
graded ring to dg modules over the ring viewed as a dg algebra with
differential zero; see \cite[\S10.9]{Maclane:1963a}, where
this functor is called condensation, and \cite[\S3.3]{Keller:1994a}.

\subsection*{Totalization}
Let $A$ be a graded algebra. For each graded $A$-module $N$ and
integer $d$ we write $N[d]$ for the graded $A$-module with
$N[d]^{i}=N^{d+i}$, and multiplication the same as the one on $N$.

Let $F$ be a complex of graded $A$-modules with differential $\delta$;
so each $F^{i}$ is a graded $A$-module, $\delta^{i}\col F^{i}\to
F^{i+1}$ are morphisms of graded $A$-modules, and
$\delta^{i+1}\delta^{i}=0$.  We write $F^{i,j}$ for the component of
degree $j$ in the graded module $F^{i}$. The \emph{totalization} of
$F$, denoted $\tot F$, is the dg abelian group with
\begin{align*}
(\tot F)^{n} &= \bigoplus_{i+j=n} F^{i,j} \quad \text{for each $n\in\bbZ$} \\ 
\dd(f) &= \delta^{i}(f) \quad \text{for each $f\in F^{i,j}$} 
\intertext{We consider $\tot F$ as a graded $A$-module with multiplication defined by} a\cdot f &= (-1)^{di}af \quad \text{for each $a\in A^{d}$ and $f\in F^{i,j}$.}
\end{align*}
A routine calculation shows that $\tot F$ is then a dg $A$-module, where $A$ is viewed as dg algebra with zero differential, and that each morphism $\alpha\col F\to G$ of complexes of graded $A$-modules induces a morphism $\tot\alpha\col \tot F\to \tot G$ of dg
$A$-modules. Moreover, there are equalities of dg $A$-modules:
\begin{itemize}
\item $\tot A = A$;
\item $\tot N[d] = \Si^{d} \tot N$ for each graded $A$-module $N$ and
integer $d$;
\item $\tot \Si^{n}F = \Si^{n}\tot F$.
\end{itemize}
One thus gets an additive functor from the category of complexes of graded $A$-modules to the category of dg $A$-modules. It is easy to check that if the complex $F$ is acyclic so is $\tot F$. 

Indeed, fix a cycle $z$ in $(\tot F)^{n}$, and write $z=\sum_{i}z_{i}$ where $z_{i}\in F^{i,n-i}$. Since $\delta(z) = \sum_{i}\delta^{i}(z_{i})$ and $\delta^{i}(z_{i})\in F^{i+1,n-i}$, each $z_{i}$ is a cycle in $F^{i}$. Since $F$ is acyclic there exist elements $w_{i}\in F^{i-1,j}$ with $\delta^{i-1}(w_{i})=z_{i}$; moreover, one may take $w_{i}=0$ when $z_{i}=0$. Note that the element $w=\sum_{i}w_{i}$ is in $(\tot F)^{n-1}$ and $\delta(w)=z$. 

In conclusion, $\tot$ induces an exact functor
\[
\tot\col \sfD(\GrMod A)\lto \sfD(A)\,.
\]
of triangulated categories; here $\sfD(\GrMod A)$ is the derived category of graded $A$-modules, while $\sfD(A)$ is the derived category of dg $A$-modules.

\begin{lemma}
\label{lem:tot}
Let $E$ be the Koszul complex on a sequence $\bsa=a_{1},\dots,a_{c}$
of homogenous central elements in $A$. Then $\tot E \cong \Si^{d}\kos
A{\bsa}$ in $\sfD(A)$, where $d=\sum_{n} |a_{n}|$.
\end{lemma}

\begin{proof}
Indeed, since $\tot$ preserves exact triangles, and both $E$ and $\kos
A{\bsa}$ can be obtained as iterated mapping cones, it suffices to
verify the statement for the Koszul complex on a single element, say
$a$.  The desired result is then immediate from the properties of
$\tot$ listed above.
\end{proof}

We require also some elementary results concerning transfer of
stratification along exact functors; a detailed study is taken up in
\cite[Section 7]{\bik:2009b}.

\subsection*{Change of categories}
As before $R$ is a graded commutative noetherian ring and $\sfT$ is a
compactly generated $R$-linear triangulated category. Let $F\col
\sfU\to \sfT$ be an equivalence of triangulated categories. Observe
that $\sfU$ is then compactly generated; it is also $R$-linear with
action given by the isomorphism of graded abelian groups
\[
\Hom^{*}_{\sfU}(X,Y)\cong \Hom^{*}_{\sfT}(F X,F Y)
\]
induced by $F$, for all $X,Y$ in $\sfU$.

\begin{proposition}
\label{prop:cot}
The ring $R$ stratifies $\sfU$ if and only if it stratifies $\sfT$.
\end{proposition}

\begin{proof}
Using \cite[Corollary 5.9]{\bik:2008a}, it is easy to verify that for
each $\fp$ in $\Spec R$ and $X$ in $\sfU$, there is an isomorphism
$F(\gam_{\fp}X)\cong \gam_{\fp}(F X)$, and that the induced functor
$\gam_{\fp}\sfU\to \gam_{\fp}\sfT$ is an equivalence of triangulated
categories. Given this, it is immediate from definitions that $R$
stratifies $\sfU$ if and only if it stratifies $\sfT$.
\end{proof}

When $A\to B$ is a quasi-isomorphism of dg algebras,
$B\lotimes_{A}-\col \sfD(A)\to \sfD(B)$ is an equivalence of
categories, with quasi-inverse the restriction of scalars; see, for
example, \cite[3.6]{Avramov/Buchweitz/Iyengar/Miller:2009a}, or
\cite[6.1]{Keller:1994a}. The preceding result thus yields:

\begin{corollary}
\label{cor:changeofdgas}
Let $A$ and $B$ be quasi-isomorphic dg algebras. If $\sfD(A)$ is
stratified by an action of $R$, then $\sfD(B)$ is stratified by the
induced $R$-action.\qed
\end{corollary}

\subsection*{Proof of Theorem~\ref{thm:cdga}}
Let $R=\hh A$. The category $\sfD(A)$ is tensor triangulated so it admits an $R$-action induced by the isomorphism $R\cong \Hom^{*}_{\sfD(A)}(A,A)$. The dg algebras $A$ and $\hh A$ are quasi-isomorphic, as $A$ is formal, so it suffices to prove that $\sfD(\hh A)$ is stratified by the induced $R$-action; see Corollary~\ref{cor:changeofdgas}. It is easy to verify that the homomorphism $R\to \Hom^{*}_{\sfD(\hh A)}(\hh A,\hh A)=\hh A$ induced by this $R$-action is bijective, and hence that $\sfD(\hh A)$ is
stratified by $R$ if and only if it is stratified by the canonical
$\hh A$-action.

In summary, replacing $A$ by $\hh A$ we may thus assume the
differential of $A$ is zero. Set $\sfD=\sfD(A)$. Since $A$ is a unit
and a generator of this tensor triangulated category, its localizing
subcategories are tensor closed. The local-global principle then holds
for $\sfD$, by Theorem~\ref{thm:ttlg}.  It remains to verify
stratification condition (S2).

Fix a $\fp$ in $\Spec A$. Since $A$ is a compact generator for $\sfD$,
a dg $A$-module $M$ is in $\gam_{\fp}\sfD$ if and only if the
$A$-module $H^*(M)=\Hom^{*}_{\sfD}(A,M)$ is $\fp$-local and
$\fp$-torsion. Hence for such an $M$ the localization map $M\to
M_{\fp}$ is an isomorphism; here $M_{\fp}$ denotes the usual
(homogenous) localization of $M$ at $\fp$. Localizing $A$ at $\fp$ we
may thus assume that it is local with maximal ideal $\fp$; set
$k=A/\fp$, which is a graded field. Setting $\mcV=\mcV(\fp)$, one has
an isomorphism of functors $\gam_{\fp}\cong \gam_{\mcV}$.

Evidently, $k$ is in $\gam_{\mcV}\sfD$, so to verify condition (S2) it suffices to verify that
\begin{equation}
\label{eq:cdga}
\Loc_{\sfD}(M) = \Loc_{\sfD}(k)
\end{equation}
holds for each $M$ in $\gam_{\mcV}\sfD$ with $H^{*}(M)\ne 0$.

It is enough to prove that \eqref{eq:cdga} holds for $M=\gam_{\mcV}A$. Indeed, applying the functor $-\lotimes_{A}M$ would then yield the second equality below:
\[
\Loc_{\sfD}(M) = \Loc_{\sfD}(\gam_{\mcV}A\lotimes_AM) =
\Loc_{\sfD}(k\lotimes_{A}M)\,,
\]
while the first one holds, by \eqref{eq:loc-tensor}, since $M\cong \gam_{\mcV}M$; in particular, $H^{*}(k\lotimes_{A}M)\ne 0$. Since $k$ is a graded field and the action of $A$ on $k\lotimes_{A}M$ factors
through $k$, this implies $\Loc_{\sfD}(k\lotimes_{A}M)=\Loc_{\sfD}(k)$. Combining with the equality above gives \eqref{eq:cdga}.

Now we verify \eqref{eq:cdga} for $M=\gam_{\mcV}A$. The dg module $k$ is isomorphic to $\gam_{\mcV}A\lotimes_{A}k$ and hence in $\Loc_{\sfD}(\gam_{\mcV}A)$. It remains to prove that $\gam_{\mcV}A$ is in $\Loc_{\sfD}(k)$. Let $\bsa = a_{1},\dots,a_{c}$ be a homogeneous set of generators for the ideal $\fp$, and let $\bsa^{2}$ denote the sequence $a_{1}^{2},\dots,a_{c}^{2}$. It suffices to prove that
\begin{equation}
\label{eq:cdga2}
\kos A{\bsa^{2}} \in \Thick_{\sfD}(k)\,,
\end{equation} 
for then one has 
\[
\Loc_{\sfD}(\kos A{\fp}) = \Loc_{\sfD}(\gam_{\mcV}A) = \Loc_{\sfD}(\gam_{\mcV(\bsa^{2})}A) 
= \Loc_{\sfD}(\kos A{\bsa^{2}}) \subseteq \Loc_{\sfD}(k)
\]
where the first and third equalities are by  Proposition~\ref{prop:kosproperties}, and the second holds because the radical of the ideal $(\bsa^{2})$ equals $\fp$, so that $\mcV(\bsa^{2})=\mcV$. 
 
Let $\tot\col\sfD(\GrMod A)\to \sfD$ be the totalization functor described above and $E$ in $\sfD(\GrMod A)$ the Koszul complex on the sequence $\bsa^{2}$; note that the elements $a_{i}$ are central in $A$, since they are of even degree. The complex $E$ is bounded, consists of finitely generated graded $A$-modules, and satisfies $(\bsa^{2})\cdot \hh E=0$. Since $k$ is a graded field, the subquotients of the filtration $\{0\}\subseteq (\bsa)\hh E\subseteq \hh E$ are thus finite direct sums of shifts of $k$. Hence there are inclusions
\[
E\in\Thick(\hh E)\subseteq \Thick(k)
\]
in $\sfD(\GrMod A)$; see, for example, \cite[Theorem~6.2(3)]{Avramov/Buchweitz/Iyengar/Miller:2009a}.  Since
$\tot$ is an exact functor, it follows that $\tot E$ is in $\Thick(\tot k)$ in $\sfD$. It remains to note that $\tot k = k$ and that $\tot E$ is isomorphic to a suspension of $\kos A{\bsa^{2}}$, by Lemma~\ref{lem:tot}.

This justifies \eqref{eq:cdga2} and hence completes the proof of the theorem. \qed

\bibliographystyle{amsplain}

\end{document}